\documentclass[10pt]{article}
\usepackage{amsmath}
\usepackage{amssymb}
\usepackage{amsthm}
\usepackage{mathrsfs}
\numberwithin{equation}{section}

\begin{document}

\title{Two-weight, weak type norm inequalities for fractional integral operators and commutators on weighted Morrey and amalgam spaces}
\author{Hua Wang \footnote{E-mail address: wanghua@pku.edu.cn.}\\
\footnotesize{College of Mathematics and Econometrics, Hunan University, Changsha 410082, P. R. China}\\
\footnotesize{\&~Department of Mathematics and Statistics, Memorial University, St. John's, NL A1C 5S7, Canada}}
\date{}
\maketitle

\begin{abstract}
Let $0<\gamma<n$ and $I_\gamma$ be the fractional integral operator of order $\gamma$, $I_{\gamma}f(x)=\int_{\mathbb R^n}|x-y|^{\gamma-n}f(y)\,dy$, and let $[b,I_\gamma]$ be the linear commutator generated by a symbol function $b$ and $I_\gamma$, $[b,I_{\gamma}]f(x)=b(x)\cdot I_{\gamma}f(x)-I_\gamma(bf)(x)$. This paper is concerned with two-weight, weak type norm estimates for such operators on the weighted Morrey and amalgam spaces. Based on weak-type norm inequalities on weighted Lebesgue spaces and certain $A_p$-type conditions on pairs of weights, we can establish the weak-type norm inequalities for fractional integral operator $I_{\gamma}$ as well as the corresponding commutator in the framework of weighted Morrey and amalgam spaces. Furthermore, some estimates for the extreme case are also obtained on these weighted spaces.\\
MSC(2010): 42B20; 46E30; 47B38; 47G10\\
Keywords: Fractional integral operators; weighted Morrey spaces; weighted amalgam spaces; commutators; weak-type norm inequalities.
\end{abstract}

\section{Introduction}

\subsection{Fractional integral operators}

\newtheorem{theorem}{Theorem}[section]

\newtheorem{lemma}{Lemma}[section]

\newtheorem{corollary}[theorem]{Corollary}

Let $\mathbb R^n$ be the $n$-dimensional Euclidean space equipped with the Euclidean norm $|\cdot|$ and the Lebesgue measure $dx$. For given $\gamma$, $0<\gamma<n$, the fractional integral operator (or Riesz potential) $I_{\gamma}$ with order $\gamma$ is defined by (see \cite{stein} for the basic properties of $I_{\gamma}$)
\begin{equation}\label{frac}
I_{\gamma}f(x):=\frac{1}{\zeta(\gamma)}\int_{\mathbb R^n}\frac{f(y)}{|x-y|^{n-\gamma}}\,dy
\qquad\&\qquad \zeta(\gamma)=\frac{\pi^{\frac{n}{\,2\,}}2^\gamma\Gamma(\frac{\gamma}{\,2\,})}{\Gamma(\frac{n-\gamma}{2})}.
\end{equation}
Weighted norm inequalities for fractional integral operators arise naturally in harmonic analysis, and have been extensively studied by several authors. The study of two-weight problem for $I_{\gamma}$ was initiated by Sawyer in his pioneer paper \cite{sawyer2}. By a weight $w$, we mean that $w$ is a nonnegative and locally integrable function. In \cite{sawyer2}, Sawyer concerned the following question: Suppose that $1<p\leq q<\infty$. For which pairs of weights $(w,\nu)$ on $\mathbb R^n$ is the fractional integral operator bounded from $L^p(\nu)$ into weak-$L^q(w)$? A necessary and sufficient condition for the weak-type $(p,q)$ inequality was given by Sawyer. More specifically, he showed that
\begin{theorem}[\cite{sawyer2}]\label{S}
Let $0<\gamma<n$ and $1<p\leq q<\infty$. Given a pair of weights $(w,\nu)$ on $\mathbb R^n$, the weak-type inequality
\begin{equation*}
\sigma\cdot w\big(\big\{x\in\mathbb R^n:\big|I_\gamma f(x)\big|>\sigma\big\}\big)^{1/q}
\leq C\bigg(\int_{\mathbb R^n}|f(x)|^p\nu(x)\,dx\bigg)^{1/p}
\end{equation*}
holds for any $\sigma>0$ if and only if
\begin{equation*}
\bigg(\int_Q\big[I_{\gamma}(\chi_Q w)(x)\big]^{p'}\nu(x)^{1-p'}dx\bigg)^{1/{p'}}\leq C\cdot w(Q)^{1/{q'}}<\infty\qquad(\dag)
\end{equation*}
for all cubes $Q$ in $\mathbb R^n$. Here $\chi_Q$ denotes the characteristic function of the cube $Q$, $p'=p/{(p-1)}$ denotes the conjugate index of $p$, and $C$ is a universal constant.
\end{theorem}
Sawyer's result is interesting and important, and it promotes a series of research works on this subject (see e.g. \cite{cruz2,cruz3,li,liu,martel,sawyer,zhang}), but it has the defect that the condition $(\dag)$ involves the fractional integral operator $I_{\gamma}$ itself.
In \cite{cruz2}, Cruz-Uribe and P\'erez considered the case when $q=p$, and found a sufficient $A_p$-type condition on a pair of weights $(w,\nu)$ which ensures the boundedness of the operator $I_{\gamma}$ from $L^p(\nu)$ into weak-$L^p(w)$, where $1<p<\infty$. The condition $(\S')$ given below is simpler than $(\dag)$ in the sense that it does not involve the operator $I_{\gamma}$ itself, and hence it can be more easily verified.
\begin{theorem}[\cite{cruz2}]\label{Two3}
Let $0<\gamma<n$ and $1<p<\infty$. Given a pair of weights $(w,\nu)$ on $\mathbb R^n$, suppose that for some $r>1$ and for all cubes $Q$ in $\mathbb R^n$,
\begin{equation*}
\big|Q\big|^{\gamma/n}\cdot\left(\frac{1}{|Q|}\int_Q w(x)^r\,dx\right)^{1/{(rp)}}\left(\frac{1}{|Q|}\int_Q \nu(x)^{-p'/p}\,dx\right)^{1/{p'}}\leq C<\infty.\qquad(\S')
\end{equation*}
Then the fractional integral operator $I_\gamma$ satisfies the weak-type $(p,p)$ inequality
\begin{equation*}
\sigma\cdot w\big(\big\{x\in\mathbb R^n:\big|I_\gamma f(x)\big|>\sigma\big\}\big)^{1/p}
\leq C\bigg(\int_{\mathbb R^n}|f(x)|^p\nu(x)\,dx\bigg)^{1/p},\quad\mbox{for any }~\sigma>0,
\end{equation*}
where $C$ does not depend on $f$ nor on $\sigma>0$.
\end{theorem}
The proof of Theorem \ref{Two3} is quite complicated. It depends on an inequality relating the Hardy--Littlewood maximal function and the sharp maximal function which is strongly reminiscent of the good-$\lambda$ inequality of Fefferman and Stein. For another, more elementary proof, see also \cite{cruz3}. This solves a problem posed by Sawyer and Wheeden in \cite{sawyer}. Moreover, in \cite{li}, Li improved this result by replacing the ``power bump" in $(\S')$ by a smaller ``Orlicz bump"(see also \cite{cruz4}). On the other hand, for given $0<\gamma<n$, the linear commutator $[b,I_{\gamma}]$ generated by a suitable function $b$ and $I_{\gamma}$ is defined by
\begin{align}\label{lfrac}
[b,I_{\gamma}]f(x)&:=b(x)\cdot I_{\gamma}f(x)-I_\gamma(bf)(x)\notag\\
&=\frac{1}{\zeta(\gamma)}\int_{\mathbb R^n}\frac{[b(x)-b(y)]\cdot f(y)}{|x-y|^{n-\gamma}}\,dy.
\end{align}
This commutator was first introduced by Chanillo in \cite{chanillo}. In \cite{liu}, Liu and Lu obtained a sufficient $A_p$-type condition for the commutator $[b,I_{\gamma}]$ to satisfy a two-weight weak type $(p,p)$ inequality, where $1<p<\infty$. That condition is an $A_p$-type condition in the scale of Orlicz spaces (see $(\S\S')$ given below).
\begin{theorem}[\cite{liu}]\label{Two4}
Let $0<\gamma<n$, $1<p<\infty$ and $b\in \mathrm{BMO}(\mathbb R^n)$. Given a pair of weights $(w,\nu)$ on $\mathbb R^n$, suppose that for some $r>1$ and for all cubes $Q$ in $\mathbb R^n$,
\begin{equation*}
\big|Q\big|^{\gamma/n}\cdot\left(\frac{1}{|Q|}\int_Q w(x)^r\,dx\right)^{1/{(rp)}}\big\|\nu^{-1/p}\big\|_{\mathcal A,Q}\leq C<\infty,\qquad(\S\S')
\end{equation*}
where $\mathcal A(t)=t^{p'}(1+\log^+t)^{p'}$ and $\log^+t:=\max\{\log t,0\},$ that is,
\begin{equation*}
\log^+t=
\begin{cases}
\log t,\  &\mbox{as}~~t>1;\\
0,\  &\mbox{otherwise}.\\
\end{cases}
\end{equation*}
Then the linear commutator $[b,I_\gamma]$ satisfies the weak-type $(p,p)$ inequality
\begin{equation*}
\sigma\cdot w\big(\big\{x\in\mathbb R^n:\big|[b,I_\gamma]f(x)\big|>\sigma\big\}\big)^{1/p}
\leq C\bigg(\int_{\mathbb R^n}|f(x)|^p\nu(x)\,dx\bigg)^{1/p},\quad\mbox{for any }~\sigma>0,
\end{equation*}
where $C$ does not depend on $f$ nor on $\sigma>0$.
\end{theorem}
In \cite{martel}, Martell considered the case when $q>p$, and gave a verifiable condition which is sufficient for the two-weight, weak type $(p,q)$ inequality for fractional integral operator $I_{\gamma}$. The condition $(\S)$ given below (in the Euclidean setting of \cite{martel}) is also simpler than the one in Theorem \ref{S}.
\begin{theorem}[\cite{martel}]\label{Two1}
Let $0<\gamma<n$ and $1<p<q<\infty$. Given a pair of weights $(w,\nu)$ on $\mathbb R^n$, suppose that for some $r>1$ and for all cubes $Q$ in $\mathbb R^n$,
\begin{equation*}
\big|Q\big|^{\gamma/n+1/q-1/p}\cdot\left(\frac{1}{|Q|}\int_Q w(x)^r\,dx\right)^{1/{(rq)}}\left(\frac{1}{|Q|}\int_Q \nu(x)^{-p'/p}\,dx\right)^{1/{p'}}\leq C<\infty.\quad(\S)
\end{equation*}
Then the fractional integral operator $I_\gamma$ satisfies the weak-type $(p,q)$ inequality
\begin{equation}\label{assump1.2}
\sigma\cdot w\big(\big\{x\in\mathbb R^n:\big|I_\gamma f(x)\big|>\sigma\big\}\big)^{1/q}
\leq C\bigg(\int_{\mathbb R^n}|f(x)|^p\nu(x)\,dx\bigg)^{1/p},\quad\mbox{for any }~\sigma>0,
\end{equation}
where $C$ does not depend on $f$ nor on $\sigma>0$.
\end{theorem}
Furthermore, in \cite{zhang}, Zhang sharpened Martell's result by replacing the local $L^r$ norm on the left-hand side of $(\S)$ by the smaller Orlicz space norm. On the other hand, by comparing Theorem \ref{Two1} with Theorems \ref{Two3} and \ref{Two4}, it is natural to conjecture that when $q>p$, there is a two-weight, weak type $(p,q)$ inequality for the commutator $[b,I_\gamma]$ of fractional integral operator. By using the same method as in the proof of Theorem 1 in \cite{liu} and certain Orlicz norm, we are able to obtain the following sufficient condition on a pair of weights $(w,\nu)$ to ensure the $L^p(\nu)\rightarrow WL^q(w)$ boundedness of $[b,I_\gamma]$, whenever $b$ belongs to $\mathrm{BMO}(\mathbb R^n)$. More specifically, the following statement is true.
\begin{theorem}\label{Two2}
Let $0<\gamma<n$, $1<p<q<\infty$ and $b\in\mathrm{BMO}(\mathbb R^n)$. Given a pair of weights $(w,\nu)$ on $\mathbb R^n$, suppose that for some $r>1$ and for all cubes $Q$ in $\mathbb R^n$,
\begin{equation*}
\big|Q\big|^{\gamma/n+1/q-1/p}\cdot\left(\frac{1}{|Q|}\int_Q w(x)^r\,dx\right)^{1/{(rq)}}\big\|\nu^{-1/p}\big\|_{\mathcal A,Q}\leq C<\infty,\qquad(\S\S)
\end{equation*}
where $\mathcal A(t)=t^{p'}(1+\log^+t)^{p'}$. Then the linear commutator $[b,I_\gamma]$ satisfies the weak-type $(p,q)$ inequality
\begin{equation}\label{assump2.2}
\sigma\cdot w\big(\big\{x\in\mathbb R^n:\big|[b,I_\gamma]f(x)\big|>\sigma\big\}\big)^{1/q}
\leq C\bigg(\int_{\mathbb R^n}|f(x)|^p\nu(x)\,dx\bigg)^{1/p},\quad\mbox{for any }~\sigma>0,
\end{equation}
where $C$ does not depend on $f$ nor on $\sigma>0$.
\end{theorem}
The details are omitted here. Note that the condition $(\S\S)$ reduces to the condition $(\S\S')$ provided that $p=q$.

\emph{Question}. In view of Theorems \ref{Two3}, \ref{Two4}, \ref{Two1}, and \ref{Two2}, it is a natural and interesting problem to find some sufficient
conditions for which the two-weight, weak type norm inequalities hold for the operators $I_{\gamma}$ and $[b,I_\gamma]$, in the endpoint case $p=1$.

In this paper, we are mainly interested in the weighted Morrey spaces and weighted amalgam spaces. Let us recall their definitions.

\subsection{Weighted Morrey spaces}

\newtheorem{defn}{Definition}[section]
The classical Morrey space $L^{p,\lambda}(\mathbb R^n)$ was introduced by Morrey \cite{morrey} in connection with elliptic partial differential equations. Let $1\leq p<\infty$ and $0\leq\lambda\leq n$. We recall that a real-valued function $f$ is said to belong to the space $L^{p,\lambda}(\mathbb R^n)$ on the $n$-dimensional Euclidean space $\mathbb R^n$, if the following norm is finite:
\begin{equation*}
\|f\|_{L^{p,\lambda}}:=\sup_{(x,r)\in\mathbb R^n\times(0,\infty)}\bigg(r^{\lambda-n}\int_{B(x,r)}|f(y)|^p\,dy\bigg)^{1/p},
\end{equation*}
where $B(x,r)=\big\{y\in\mathbb R^n:|x-y|<r\big\}$ is the Euclidean ball with center $x\in\mathbb R^n$ and radius $r\in(0,\infty)$ as well as the Lebesgue measure $|B(x,r)|=v_n\cdot r^n$. Here $v_n$ is the volume of the unit ball of $\mathbb R^n$. In particular, one has
\begin{equation*}
L^{p,0}(\mathbb R^n)=L^\infty(\mathbb R^n) \qquad\&\qquad L^{p,n}(\mathbb R^n)=L^p(\mathbb R^n).
\end{equation*}

In \cite{komori}, Komori and Shirai considered the weighted case, and introduced a version of weighted Morrey space, which is a natural generalization of weighted Lebesgue space.
\begin{defn}
Let $1<p<\infty$ and $0\leq\kappa<1$. For two weights $w$ and $\nu$ on $\mathbb R^n$, the weighted Morrey space $L^{p,\kappa}(\nu,w)$ is defined by
\begin{equation*}
L^{p,\kappa}(\nu,w):=\Big\{f\in L^p_{loc}(\nu):\big\|f\big\|_{L^{p,\kappa}(\nu,w)}<\infty\Big\},
\end{equation*}
where the norm is given by
\begin{equation*}
\big\|f\big\|_{L^{p,\kappa}(\nu,w)}
:=\sup_{Q\subset\mathbb R^n}\bigg(\frac{1}{w(Q)^{\kappa}}\int_Q |f(x)|^p\nu(x)\,dx\bigg)^{1/p},
\end{equation*}
and the supremum is taken over all cubes $Q$ in $\mathbb R^n$.
\end{defn}

\begin{defn}
Let $1<p<\infty$, $0\leq\kappa<1$ and $w$ be a weight on $\mathbb R^n$. We define the weighted weak Morrey space $WL^{p,\kappa}(w)$ as the set of all measurable functions $f$ satisfying
\begin{equation*}
\big\|f\big\|_{WL^{p,\kappa}(w)}:=\sup_{Q\subset\mathbb R^n}\sup_{\sigma>0}\frac{1}{w(Q)^{{\kappa}/p}}\sigma
\cdot \Big[w\big(\big\{x\in Q:|f(x)|>\sigma\big\}\big)\Big]^{1/p}<\infty.
\end{equation*}
\end{defn}
By definition, it is clear that
\begin{equation*}
L^{p,0}(\nu,w)=L^p(\nu)\qquad \& \qquad WL^{p,0}(w)=WL^p(w).
\end{equation*}

\subsection{Weighted amalgam spaces}

Let $1\leq p,s\leq\infty$, a function $f\in L^p_{loc}(\mathbb R^n)$ is said to be in the Wiener amalgam space $(L^p,L^s)(\mathbb R^n)$ of $L^p(\mathbb R^n)$ and $L^s(\mathbb R^n)$, if the function $y\mapsto\|f(\cdot)\cdot\chi_{B(y,1)}\|_{L^p}$ belongs to $L^s(\mathbb R^n)$, where $B(y,1)$ is an open ball in $\mathbb R^n$ centered at $y$ with radius $1$, $\chi_{B(y,1)}$ is the characteristic function of the ball $B(y,1)$, and $\|\cdot\|_{L^p}$ is the usual Lebesgue norm in $L^p(\mathbb R^n)$. In \cite{fofana}, Fofana introduced a new class of function spaces $(L^p,L^s)^{\alpha}(\mathbb R^n)$ which turn out to be the subspaces of $(L^p,L^s)(\mathbb R^n)$. More precisely, for $1\leq p,s,\alpha\leq\infty$, we define the amalgam space $(L^p,L^s)^{\alpha}(\mathbb R^n)$ of $L^p(\mathbb R^n)$ and $L^s(\mathbb R^n)$ as the set of all measurable functions $f$ satisfying $f\in L^p_{loc}(\mathbb R^n)$ and $\big\|f\big\|_{(L^p,L^s)^{\alpha}(\mathbb R^n)}<\infty$, where
\begin{equation*}
\begin{split}
\big\|f\big\|_{(L^p,L^s)^{\alpha}(\mathbb R^n)}
:=&\sup_{r>0}\left\{\int_{\mathbb R^n}\Big[\big|B(y,r)\big|^{1/{\alpha}-1/p-1/s}\big\|f\cdot\chi_{B(y,r)}\big\|_{L^p(\mathbb R^n)}\Big]^sdy\right\}^{1/s}\\
=&\sup_{r>0}\Big\|\big|B(y,r)\big|^{1/{\alpha}-1/p-1/s}\big\|f\cdot\chi_{B(y,r)}\big\|_{L^p(\mathbb R^n)}\Big\|_{L^s(\mathbb R^n)},
\end{split}
\end{equation*}
with the usual modification when $p=\infty$ or $s=\infty$, and $|B(y,r)|$ is the Lebesgue measure of the ball $B(y,r)$. As it was shown in \cite{fofana} that the space $(L^p,L^s)^{\alpha}(\mathbb R^n)$ is non-trivial if and only if $p\leq\alpha\leq s$; thus in the remaining of this paper we will always assume that this condition $p\leq\alpha\leq s$ is satisfied. Let us consider the following two special cases:
\begin{enumerate}
  \item If we take $p=s$, then $p=\alpha=s$. By Fubini's theorem, it is easy to check that
\begin{equation*}
\begin{split}
&\Big\|\big|B(y,r)\big|^{-1/p}\big\|f\cdot\chi_{B(y,r)}\big\|_{L^p(\mathbb R^n)}\Big\|_{L^p(\mathbb R^n)}\\
=&\bigg[\int_{\mathbb R^n}\big|B(y,r)\big|^{-1}\bigg(\int_{\mathbb R^n}|f(x)|^p\cdot\chi_{B(y,r)}\,dx\bigg)\,dy\bigg]^{1/p}\\
=&\bigg[\int_{\mathbb R^n}\big|B(y,r)\big|^{-1}\bigg(\int_{B(x,r)}|f(x)|^p\,dy\bigg)\,dx\bigg]^{1/p}\\
=&\bigg(\int_{\mathbb R^n}|f(x)|^p\,dx\bigg)^{1/p},
\end{split}
\end{equation*}
where the last equality holds since $|B(y,r)|^{-1}\cdot|B(x,r)|=1$. Hence, the amalgam space $(L^p,L^s)^{\alpha}(\mathbb R^n)$ is equal to the Lebesgue space $L^p(\mathbb R^n)$ with the same norms provided that $p=\alpha=s$.
  \item If $s=\infty$, then we can see that in such a situation, the amalgam space $(L^p,L^s)^{\alpha}(\mathbb R^n)$ is equal to the classical Morrey space $L^{p,\lambda}(\mathbb R^n)$ with equivalent norms, where $\lambda={(pn)}/{\alpha}$.
\end{enumerate}

In this paper, we will consider the weighted version of $(L^p,L^s)^{\alpha}(\mathbb R^n)$.
\begin{defn}\label{amalgam}
Let $1\leq p\leq\alpha\leq s\leq\infty$, and let $\nu,w,\mu$ be three weights on $\mathbb R^n$. We denote by $(L^p,L^s)^{\alpha}(\nu,w;\mu)$ the weighted amalgam space, the space of all locally integrable functions $f$ such that
\begin{equation*}
\begin{split}
\big\|f\big\|_{(L^p,L^s)^{\alpha}(\nu,w;\mu)}
:=&\sup_{\ell>0}\left\{\int_{\mathbb R^n}\Big[w(Q(y,\ell))^{1/{\alpha}-1/p-1/s}\big\|f\cdot\chi_{Q(y,\ell)}\big\|_{L^p(\nu)}\Big]^s\mu(y)\,dy\right\}^{1/s}\\
=&\sup_{\ell>0}\Big\|w(Q(y,\ell))^{1/{\alpha}-1/p-1/s}\big\|f\cdot\chi_{Q(y,\ell)}\big\|_{L^p(\nu)}\Big\|_{L^s(\mu)}<\infty,
\end{split}
\end{equation*}
with $w(Q(y,\ell))=\displaystyle\int_{Q(y,\ell)}w(x)\,dx$ and the usual modification when $s=\infty$.
\end{defn}
\begin{defn}
Let $1\leq p\leq\alpha\leq s\leq\infty$, and let $w,\mu$ be two weights on $\mathbb R^n$. We denote by $(WL^p,L^s)^{\alpha}(w;\mu)$ the weighted weak amalgam space consisting of all measurable functions $f$ such that
\begin{equation*}
\begin{split}
\big\|f\big\|_{(WL^p,L^s)^{\alpha}(w;\mu)}
:=&\sup_{\ell>0}\left\{\int_{\mathbb R^n}\Big[w(Q(y,\ell))^{1/{\alpha}-1/p-1/s}\big\|f\cdot\chi_{Q(y,\ell)}\big\|_{WL^p(w)}\Big]^s\mu(y)\,dy\right\}^{1/s}\\
=&\sup_{\ell>0}\Big\|w(Q(y,\ell))^{1/{\alpha}-1/p-1/s}\big\|f\cdot\chi_{Q(y,\ell)}\big\|_{WL^p(w)}\Big\|_{L^s(\mu)}<\infty,
\end{split}
\end{equation*}
with $w(Q(y,\ell))=\displaystyle\int_{Q(y,\ell)}w(x)\,dx$ and the usual modification when $s=\infty$.
\end{defn}
Note that in the particular case when $\mu\equiv1$, this kind of weighted (weak) amalgam space was introduced by Feuto in \cite{feuto2} (see also \cite{feuto1}). We remark that Feuto \cite{feuto2} considered ball $B$ instead of cube $Q$ in his definition, but these two definitions are evidently equivalent. Also note that when $1\leq p\leq\alpha$ and $s=\infty$, then $(L^p,L^s)^{\alpha}(\nu,w;\mu)$ is just the weighted Morrey space $L^{p,\kappa}(\nu,w)$ with $\kappa=1-p/{\alpha}$, and $(WL^p,L^s)^{\alpha}(w;\mu)$ is just the weighted weak Morrey space $W L^{p,\kappa}(w)$ with $\kappa=1-p/{\alpha}$.

Recently, in \cite{wang1,wang3,wang4}, the author studied the two-weight, weak-type $(p,p)$ inequalities for fractional integral operator, as well as its commutators on weighted Morrey and amalgam spaces, under some $A_p$-type conditions $(\S')$ and $(\S\S')$ on the pair $(w,\nu)$. As a continuation of the works mentioned above, in this paper, we consider related problems about two-weight, weak type $(p,q)$ inequalities for $I_{\gamma}$ and $[b,I_{\gamma}]$, under some other $A_p$-type conditions $(\S)$ and $(\S\S)$ on $(w,\nu)$ and $1<p<q$.

\section{Statement of our main results}
We are now in a position to state our main results. Let $p'$ be the conjugate index of $p$ whenever $p>1$; that is, $1/p+1/{p'}=1$. First we give the two-weight, weak-type norm inequalities for the fractional integral operator in the setting of weighted Morrey and amalgam spaces.

\begin{theorem}\label{mainthm:1}
Let $0<\gamma<n$, $1<p<q<\infty$ and $0<\kappa<p/q$. Given a pair of weights $(w,\nu)$ on $\mathbb R^n$, suppose that for some $r>1$ and for all cubes $Q$ in $\mathbb R^n$,
\begin{equation*}
\big|Q\big|^{\gamma/n+1/q-1/p}\cdot\bigg(\frac{1}{|Q|}\int_Q w(x)^r\,dx\bigg)^{1/{(rq)}}\bigg(\frac{1}{|Q|}\int_Q \nu(x)^{-p'/p}\,dx\bigg)^{1/{p'}}\leq C<\infty.\quad(\S)
\end{equation*}
If $w\in\Delta_2$, then the fractional integral operator $I_\gamma$ is bounded from $L^{p,\kappa}(\nu,w)$ into $WL^{q,{(\kappa q)}/p}(w)$.
\end{theorem}
\begin{theorem}\label{mainthm:2}
Let $0<\gamma<n$, $1<p<q<\infty$ and $\mu\in\Delta_2$. Given a pair of weights $(w,\nu)$ on $\mathbb R^n$, assume that for some $r>1$ and for all cubes $Q$ in $\mathbb R^n$,
\begin{equation*}
\big|Q\big|^{\gamma/n+1/q-1/p}\cdot\bigg(\frac{1}{|Q|}\int_Q w(x)^r\,dx\bigg)^{1/{(rq)}}\bigg(\frac{1}{|Q|}\int_Q \nu(x)^{-p'/p}\,dx\bigg)^{1/{p'}}\leq C<\infty.\quad(\S)
\end{equation*}
If $p\leq\alpha<\beta<s\leq\infty$ and $w\in \Delta_2$, then the fractional integral operator $I_{\gamma}$ is bounded from $(L^p,L^s)^{\alpha}(\nu,w;\mu)$ into $(WL^q,L^s)^{\beta}(w;\mu)$ with $1/{\beta}=1/{\alpha}-(1/p-1/q)$.
\end{theorem}
Next we introduce the definition of the space of $\mathrm{BMO}(\mathbb R^n)$ (see \cite{john}). Suppose that $b\in L^1_{loc}(\mathbb R^n)$, and let
\begin{equation*}
\|b\|_*:=\sup_{Q\subset\mathbb R^n}\frac{1}{|Q|}\int_Q|b(x)-b_Q|\,dx<\infty,
\end{equation*}
where $b_Q$ denotes the mean value of $b$ on $Q$, namely,
\begin{equation*}
b_Q:=\frac{1}{|Q|}\int_Q b(y)\,dy
\end{equation*}
and the supremum is taken over all cubes $Q$ in $\mathbb R^n$. Define
\begin{equation*}
\mathrm{BMO}(\mathbb R^n):=\big\{b\in L^1_{loc}(\mathbb R^n):\|b\|_*<\infty\big\}.
\end{equation*}
If we regard two functions whose difference is a constant as one, then the space $\mathrm{BMO}(\mathbb R^n)$ is a Banach space with respect to the norm $\|\cdot\|_*$. Concerning the two-weight weak-type estimates for the linear commutator $[b,I_\gamma]$ in the context of weighted Morrey and amalgam spaces, we have the following results.
\begin{theorem}\label{mainthm:3}
Let $0<\gamma<n$, $1<p<q<\infty$, $0<\kappa<p/q$ and $b\in\mathrm{BMO}(\mathbb R^n)$. Given a pair of weights $(w,\nu)$ on $\mathbb R^n$, suppose that for some $r>1$ and for all cubes $Q$ in $\mathbb R^n$,
\begin{equation*}
\big|Q\big|^{\gamma/n+1/q-1/p}\cdot\bigg(\frac{1}{|Q|}\int_Q w(x)^r\,dx\bigg)^{1/{(rq)}}\big\|\nu^{-1/p}\big\|_{\mathcal A,Q}\leq C<\infty,\quad(\S\S)
\end{equation*}
where $\mathcal A(t)=t^{p'}(1+\log^+t)^{p'}$. If $w\in A_\infty$, then the linear commutator $[b,I_\gamma]$ is bounded from $L^{p,\kappa}(\nu,w)$ into $WL^{q,{(\kappa q)}/p}(w)$.
\end{theorem}

\begin{theorem}\label{mainthm:4}
Let $0<\gamma<n$, $1<p<q<\infty$, $\mu\in\Delta_2$ and $b\in\mathrm{BMO}(\mathbb R^n)$. Given a pair of weights $(w,\nu)$ on $\mathbb R^n$, assume that for some $r>1$ and for all cubes $Q$ in $\mathbb R^n$,
\begin{equation*}
\big|Q\big|^{\gamma/n+1/q-1/p}\cdot\bigg(\frac{1}{|Q|}\int_Q w(x)^r\,dx\bigg)^{1/{(rq)}}\big\|\nu^{-1/p}\big\|_{\mathcal A,Q}\leq C<\infty,\quad(\S\S)
\end{equation*}
where $\mathcal A(t)=t^{p'}(1+\log^+t)^{p'}$. If $p\leq\alpha<\beta<s\leq\infty$ and $w\in A_\infty$, then the linear commutator $[b,I_{\gamma}]$ is bounded from $(L^p,L^s)^{\alpha}(\nu,w;\mu)$ into $(WL^q,L^s)^{\beta}(w;\mu)$ with $1/{\beta}=1/{\alpha}-(1/p-1/q)$.
\end{theorem}

Moreover, for the extreme case $\kappa=p/q$ of Theorem \ref{mainthm:1}, we will prove the following theorem, which could be viewed as a supplement of Theorem \ref{mainthm:1}.
\begin{theorem}\label{mainthm:5}
Let $0<\gamma<n$ and $1<p<q<\infty$. Given a pair of weights $(w,\nu)$ on $\mathbb R^n$, suppose that for some $r>1$ and for all cubes $Q$ in $\mathbb R^n$,
\begin{equation*}
\big|Q\big|^{\gamma/n+1/q-1/p}\cdot\bigg(\frac{1}{|Q|}\int_Q w(x)^r\,dx\bigg)^{1/{(rq)}}\bigg(\frac{1}{|Q|}\int_Q \nu(x)^{-p'/p}\,dx\bigg)^{1/{p'}}\leq C<\infty.\quad(\S)
\end{equation*}
If $\kappa=p/q$ and $w\in\Delta_2$, then the fractional integral operator $I_\gamma$ is bounded from $L^{p,\kappa}(\nu,w)$ into $\mathrm{BMO}$.
\end{theorem}
In addition, we will also discuss the extreme case $\beta=s$ of Theorem \ref{mainthm:2}. In order to do so, we need to introduce the following new $\mathrm{BMO}$-type space.
\begin{defn}
Let $1\leq s\leq\infty$ and $\mu\in\Delta_2$. The space $(\mathrm{BMO},L^s)(\mu)$ is defined as the set of all locally integrable functions $f$ satisfying $\|f\|_{**}<\infty$, where
\begin{equation}\label{BMOw}
\|f\|_{**}:=\sup_{\ell>0}\bigg\|\frac{1}{|Q(y,\ell)|}\int_{Q(y,\ell)}\big|f(x)-f_{Q(y,\ell)}\big|\,dx\bigg\|_{L^s(\mu)}.
\end{equation}
Here the $L^s(\mu)$-norm is taken with respect to the variable $y$. We also use the notation $f_{Q(y,\ell)}$ to denote the mean value of $f$ on $Q(y,\ell)$.
\end{defn}
Observe that if $s=\infty$, then $(\mathrm{BMO},L^s)(\mu)$ is just the classical $\mathrm{BMO}$ space given above.

Now we can show that $I_{\gamma}$ is bounded from $(L^p,L^s)^{\alpha}(\nu,w;\mu)$ into our new $\mathrm{BMO}$-type space defined above. This new result may be viewed as a supplement of Theorem \ref{mainthm:2}.

\begin{theorem}\label{mainthm:6}
Let $0<\gamma<n$, $1<p<q<\infty$ and $\mu\in\Delta_2$. Given a pair of weights $(w,\nu)$ on $\mathbb R^n$, assume that for some $r>1$ and for all cubes $Q$ in $\mathbb R^n$,
\begin{equation*}
\big|Q\big|^{\gamma/n+1/q-1/p}\cdot\bigg(\frac{1}{|Q|}\int_Q w(x)^r\,dx\bigg)^{1/{(rq)}}\bigg(\frac{1}{|Q|}\int_Q \nu(x)^{-p'/p}\,dx\bigg)^{1/{p'}}\leq C<\infty.\quad(\S)
\end{equation*}
If $p\leq\alpha<s\leq\infty$, $1/s=1/{\alpha}-(1/p-1/q)$ and $w\in \Delta_2$, then the fractional integral operator $I_{\gamma}$ is bounded from $(L^p,L^s)^{\alpha}(\nu,w;\mu)$ into $(\mathrm{BMO},L^s)(\mu)$.
\end{theorem}
\newtheorem{remark}[theorem]{Remark}

\section{Notation and definitions}
In this section, we recall some standard definitions and notation.
\subsection{Weights}
For given $y\in\mathbb R^n$ and $\ell>0$, we denote by $Q(y,\ell)$ the cube centered at $y$ and has side length $\ell>0$, and all cubes are assumed to have their sides parallel to the coordinate axes. Given a cube $Q(y,\ell)$ and $\lambda>0$, $\lambda Q(y,\ell)$ stands for the cube concentric with $Q$ and having side length $\lambda\sqrt{n}$ times as long, i.e., $\lambda Q(y,\ell):=Q(y,\lambda\sqrt{n}\ell)$. A non-negative function $w$ defined on $\mathbb R^n$ will be called a weight if it is locally integrable. For any given weight $w$ and any Lebesgue measurable set $E$ of $\mathbb R^n$, we denote the characteristic function of $E$ by $\chi_E$, the Lebesgue measure of $E$ by $|E|$ and the weighted measure of $E$ by $w(E)$, where $w(E):=\displaystyle\int_E w(x)\,dx$. We also denote $E^c:=\mathbb R^n\backslash E$ the complement of $E$. Given a weight $w$, we say that $w$ satisfies the \emph{doubling condition}, if there exists a finite constant $C>0$ such that for any cube $Q$ in $\mathbb R^n$, we have
\begin{equation}\label{weights}
w(2Q)\leq C\cdot w(Q).
\end{equation}
When $w$ satisfies this condition \eqref{weights}, we denote $w\in\Delta_2$ for brevity. A weight $w$ is said to belong to the Muckenhoupt's class $A_p$ for $1<p<\infty$, if there exists a constant $C>0$ such that
\begin{equation*}
\left(\frac1{|Q|}\int_Q w(x)\,dx\right)^{1/p}\left(\frac1{|Q|}\int_Q w(x)^{-p'/p}\,dx\right)^{1/{p'}}\leq C
\end{equation*}
holds for every cube $Q$ in $\mathbb R^n$. The class $A_{\infty}$ is defined as the union of the $A_p$ classes for $1<p<\infty$, i.e., $A_\infty=\bigcup_{1<p<\infty}A_p$. If $w$ is an $A_{\infty}$ weight, then we have $w\in\Delta_2$ (see \cite{garcia}).
Moreover, this class $A_\infty$ is characterized as the class of all weights satisfying the following property: there exists a number $\delta>0$ and a finite constant $C>0$ such that (see \cite{garcia})
\begin{equation}\label{compare}
\frac{w(E)}{w(Q)}\le C\left(\frac{|E|}{|Q|}\right)^\delta
\end{equation}
holds for every cube $Q\subset\mathbb R^n$ and all measurable subsets $E$ of $Q$. Given a weight $w$ on $\mathbb R^n$ and for $1\leq p<\infty$, the weighted Lebesgue space $L^p(w)$ is defined to be the collection of all measurable functions $f$ satisfying
\begin{equation*}
\big\|f\big\|_{L^p(w)}:=\bigg(\int_{\mathbb R^n}|f(x)|^pw(x)\,dx\bigg)^{1/p}<\infty.
\end{equation*}
For a weight $w$ and $1\leq p<\infty$, define the distribution function of $f$ with $w$ by
\begin{equation*}
d_f(\lambda)=w\big(\big\{x\in\mathbb R^n:|f(x)|>\lambda\big\}\big),
\end{equation*}
where $\lambda$ is a positive number. We say that $f$ is in the weighted weak Lebesgue space $WL^p(w)$, if there exists a constant $C>0$ such that
\begin{equation*}
\big\|f\big\|_{WL^p(w)}:=
\sup_{\lambda>0}\lambda\cdot d_f(\lambda)^{1/p}\leq C<\infty.
\end{equation*}

\subsection{Orlicz spaces}
We next recall some basic facts from the theory of Orlicz spaces needed for the proofs of the main results. For more information about these spaces the reader may consult the book \cite{rao}. Let $\mathcal A:[0,+\infty)\rightarrow[0,+\infty)$ be a Young function. That is, a continuous, convex and strictly increasing function satisfying $\mathcal A(0)=0$ and $\mathcal A(t)\to +\infty$ as $t\to +\infty$. For a Young function $\mathcal A$ and a cube $Q$ in $\mathbb R^n$, we will consider the $\mathcal A$-average of a function $f$ given by the following Luxemburg norm:
\begin{equation*}
\big\|f\big\|_{\mathcal A,Q}
:=\inf\left\{\lambda>0:\frac{1}{|Q|}\int_Q\mathcal A\left(\frac{|f(x)|}{\lambda}\right)dx\leq1\right\}.
\end{equation*}
In particular, when $\mathcal A(t)=t^p$ with $1<p<\infty$, it is easy to see that
\begin{equation}\label{norm}
\big\|f\big\|_{\mathcal A,Q}=\left(\frac{1}{|Q|}\int_Q|f(x)|^p\,dx\right)^{1/p};
\end{equation}
that is, the Luxemburg norm in such a situation coincides with the normalized $L^p$ norm. The main examples that we are going to consider are $\mathcal A(t)=t^p(1+\log^+t)^p$ with $1<p<\infty$.

Throughout the paper $C$ always denotes a positive constant independent of the main parameters involved, but it may be different from line to line. We will use $A\approx B$ to denote the equivalence of $A$ and $B$; that is, there exist two positive constants $C_1$, $C_2$ independent of $A$, $B$ such that $C_1 A\leq B\leq C_2 A$.

\section{Proofs of Theorems \ref{mainthm:1} and \ref{mainthm:2}}

\begin{proof}[Proof of Theorem $\ref{mainthm:1}$]
Let $f\in L^{p,\kappa}(\nu,w)$ with $1<p<q<\infty$ and $0<\kappa<p/q$. For an arbitrary fixed cube $Q=Q(x_0,\ell)$ in $\mathbb R^n$, we decompose $f$ as
\begin{equation*}
\begin{cases}
f=f_1+f_2\in L^{p,\kappa}(\nu,w);\  &\\
f_1=f\cdot\chi_{2Q};\  &\\
f_2=f\cdot\chi_{(2Q)^c},
\end{cases}
\end{equation*}
where $2Q:=Q(x_0,2\sqrt{n}\ell)$ and $\chi_{2Q}$ denotes the characteristic function of $2Q$. For any given $\sigma>0$, we then write
\begin{equation*}
\begin{split}
&\frac{1}{w(Q)^{{(\kappa q)}/p\cdot 1/q}}\sigma\cdot \Big[w\big(\big\{x\in Q:\big|I_\gamma(f)(x)\big|>\sigma\big\}\big)\Big]^{1/q}\\
\leq &\frac{1}{w(Q)^{\kappa/p}}\sigma\cdot \Big[w\big(\big\{x\in Q:\big|I_\gamma(f_1)(x)\big|>\sigma/2\big\}\big)\Big]^{1/q}\\
&+\frac{1}{w(Q)^{\kappa/p}}\sigma\cdot \Big[w\big(\big\{x\in Q:\big|I_\gamma(f_2)(x)\big|>\sigma/2\big\}\big)\Big]^{1/q}\\
:=&I_1+I_2.
\end{split}
\end{equation*}
Let us consider the first term $I_1$. Using Theorem \ref{Two1} and the condition $w\in\Delta_2$, we have
\begin{equation*}
\begin{split}
I_1&\leq C\cdot\frac{1}{w(Q)^{\kappa/p}}\left(\int_{\mathbb R^n}|f_1(x)|^p\nu(x)\,dx\right)^{1/p}\\
&=C\cdot\frac{1}{w(Q)^{\kappa/p}}\left(\int_{2Q}|f(x)|^p\nu(x)\,dx\right)^{1/p}\\
&\leq C\big\|f\big\|_{L^{p,\kappa}(\nu,w)}\cdot\frac{w(2Q)^{\kappa/p}}{w(Q)^{\kappa/p}}\\
&\leq C\big\|f\big\|_{L^{p,\kappa}(\nu,w)}.
\end{split}
\end{equation*}
This is exactly what we want. We now deal with the second term $I_2$. Note that $|x-y|\approx|x_0-y|$, whenever $x,x_0\in Q$ and $y\in(2Q)^c$. For $0<\gamma<n$ and all $x\in Q$, using the standard technique, we can see that
\begin{align}\label{pointwise1}
\big|I_{\gamma}(f_2)(x)\big|&\leq C\int_{\mathbb R^n}\frac{|f_2(y)|}{|x-y|^{n-\gamma}}dy
\leq C\int_{(2Q)^c}\frac{|f(y)|}{|x_0-y|^{n-\gamma}}dy\notag\\
&=C\sum_{j=1}^\infty\int_{2^{j+1}Q\backslash 2^{j}Q}\frac{|f(y)|}{|x_0-y|^{n-\gamma}}dy\notag\\
&\leq C\sum_{j=1}^\infty\frac{1}{|2^{j+1}Q|^{1-\gamma/n}}\int_{2^{j+1}Q}|f(y)|\,dy.
\end{align}
This pointwise estimate \eqref{pointwise1} together with Chebyshev's inequality yields
\begin{equation*}
\begin{split}
I_2&\leq \frac{2}{w(Q)^{\kappa/p}}\cdot\left(\int_Q\big|I_\gamma(f_2)(x)\big|^qw(x)\,dx\right)^{1/q}\\
&\leq C\cdot w(Q)^{1/q-\kappa/p}\sum_{j=1}^\infty\frac{1}{|2^{j+1}Q|^{1-\gamma/n}}\int_{2^{j+1}Q}|f(y)|\,dy.
\end{split}
\end{equation*}
By using H\"older's inequality with exponent $p>1$, we can deduce that
\begin{equation*}
\begin{split}
I_2&\leq C\cdot w(Q)^{1/q-\kappa/p}
\sum_{j=1}^\infty\frac{1}{|2^{j+1}Q|^{1-\gamma/n}}\left(\int_{2^{j+1}Q}|f(y)|^p\nu(y)\,dy\right)^{1/p}\\
&\times\left(\int_{2^{j+1}Q}\nu(y)^{-p'/p}\,dy\right)^{1/{p'}}\\
&\leq C\big\|f\big\|_{L^{p,\kappa}(\nu,w)}\cdot w(Q)^{1/q-\kappa/p}\sum_{j=1}^\infty\frac{w(2^{j+1}Q)^{\kappa/p}}{|2^{j+1}Q|^{1-\gamma/n}}
\times\left(\int_{2^{j+1}Q}\nu(y)^{-p'/p}\,dy\right)^{1/{p'}}\\
&=C\big\|f\big\|_{L^{p,\kappa}(\nu,w)}\sum_{j=1}^\infty\frac{w(Q)^{1/q-\kappa/p}}{w(2^{j+1}Q)^{1/q-\kappa/p}}
\cdot\frac{w(2^{j+1}Q)^{1/q}}{|2^{j+1}Q|^{1-\gamma/n}}
\times\left(\int_{2^{j+1}Q}\nu(y)^{-p'/p}\,dy\right)^{1/{p'}}.
\end{split}
\end{equation*}
Moreover, we apply H\"older's inequality again with exponent $r>1$ to get
\begin{equation}\label{U}
w\big(2^{j+1}Q\big)=\int_{2^{j+1}Q}w(y)\,dy
\leq\big|2^{j+1}Q\big|^{1/{r'}}\left(\int_{2^{j+1}Q}w(y)^r\,dy\right)^{1/r}.
\end{equation}
This indicates that
\begin{equation*}
\begin{split}
I_2&\leq C\big\|f\big\|_{L^{p,\kappa}(\nu,w)}\sum_{j=1}^\infty\frac{w(Q)^{1/q-\kappa/p}}{w(2^{j+1}Q)^{1/q-\kappa/p}}
\cdot\frac{|2^{j+1}Q|^{1/{(r'q)}}}{|2^{j+1}Q|^{1-\gamma/n}}\\
&\times\left(\int_{2^{j+1}Q}w(y)^r\,dy\right)^{1/{(rq)}}\left(\int_{2^{j+1}Q}\nu(y)^{-p'/p}\,dy\right)^{1/{p'}}\\
&\leq C\big\|f\big\|_{L^{p,\kappa}(\nu,w)}
\times\sum_{j=1}^\infty\frac{w(Q)^{1/q-\kappa/p}}{w(2^{j+1}Q)^{1/q-\kappa/p}}.
\end{split}
\end{equation*}
The last inequality is obtained by the $A_p$-type condition $(\S)$ assumed on $(w,\nu)$. Furthermore, since $w\in\Delta_2$, we can easily check that there exists a \emph{reverse doubling constant }$D=D(w)>1$ independent of $Q$ such that (see \cite[Lemma 4.1]{komori})
\begin{equation*}
w(2Q)\geq D\cdot w(Q), \quad \mbox{for any cube }\,Q\subset\mathbb R^n,
\end{equation*}
which implies that for any positive integer $j$,
\begin{equation*}
w\big(2^{j+1}Q\big)\geq D^{j+1}\cdot w(Q)
\end{equation*}
by induction principle. Hence,
\begin{align}\label{C1}
\sum_{j=1}^\infty\frac{w(Q)^{1/q-\kappa/p}}{w(2^{j+1}Q)^{1/q-\kappa/p}}
&\leq\sum_{j=1}^\infty\left(\frac{w(Q)}{D^{j+1}\cdot w(Q)}\right)^{1/q-\kappa/p}\notag\\
&=\sum_{j=1}^\infty\left(\frac{1}{D^{j+1}}\right)^{1/q-\kappa/p}\notag\\
&\leq C,
\end{align}
where the last series is convergent since the \emph{reverse doubling constant} $D>1$ and $1/q-\kappa/p>0$. Therefore, in view of \eqref{C1}, we get
\begin{equation*}
I_2\leq C\big\|f\big\|_{L^{p,\kappa}(\nu,w)},
\end{equation*}
which is our desired inequality. Combining the above estimates for $I_1$ and $I_2$, and then taking the supremum over all cubes $Q\subset\mathbb R^n$ and all $\sigma>0$, we complete the proof of Theorem \ref{mainthm:1}.
\end{proof}

\begin{proof}[Proof of Theorem $\ref{mainthm:2}$]
Let $1<p\leq\alpha<s\leq\infty$ and $f\in(L^p,L^s)^{\alpha}(\nu,w;\mu)$ with $w\in\Delta_2$ and $\mu\in\Delta_2$. For an arbitrary point $y\in\mathbb R^n$, we set $Q=Q(y,\ell)$ for the cube centered at $y$ and of side length $\ell$.
Decompose $f$ as
\begin{equation*}
\begin{cases}
f=f_1+f_2\in (L^p,L^s)^{\alpha}(\nu,w;\mu);\  &\\
f_1=f\cdot\chi_{2Q};\  &\\
f_2=f\cdot\chi_{(2Q)^c},
\end{cases}
\end{equation*}
where $2Q=Q(y,2\sqrt{n}\ell)$. Then for given $y\in\mathbb R^n$ and $\ell>0$, we write
\begin{align}\label{Ip}
&w(Q(y,\ell))^{1/{\beta}-1/q-1/s}\big\|I_{\gamma}(f)\cdot\chi_{Q(y,\ell)}\big\|_{WL^q(w)}\notag\\
&\leq 2\cdot w(Q(y,\ell))^{1/{\beta}-1/q-1/s}\big\|I_{\gamma}(f_1)\cdot\chi_{Q(y,\ell)}\big\|_{WL^q(w)}\notag\\
&+2\cdot w(Q(y,\ell))^{1/{\beta}-1/q-1/s}\big\|I_{\gamma}(f_2)\cdot\chi_{Q(y,\ell)}\big\|_{WL^q(w)}\notag\\
&:=I_1(y,\ell)+I_2(y,\ell).
\end{align}
Let us consider the first term $I_1(y,\ell)$. According to Theorem \ref{Two1}, we get
\begin{align*}
I_1(y,\ell)&\leq 2\cdot w(Q(y,\ell))^{1/{\beta}-1/q-1/s}\big\|I_\gamma(f_1)\big\|_{WL^q(w)}\notag\\
&\leq C\cdot w(Q(y,\ell))^{1/{\beta}-1/q-1/s}\bigg(\int_{\mathbb R^n}|f_1(x)|^p\nu(x)\,dx\bigg)^{1/p}\notag\\
&=C\cdot w(Q(y,\ell))^{1/{\beta}-1/q-1/s}\bigg(\int_{Q(y,2\sqrt{n}\ell)}|f(x)|^p\nu(x)\,dx\bigg)^{1/p}.\notag\\
\end{align*}
Observe that
\begin{equation}\label{WI}
1/{\beta}-1/q-1/s=1/{\alpha}-1/p-1/s
\end{equation}
is valid by our assumption $1/{\beta}=1/{\alpha}-(1/p-1/q)$. This allows us to obtain
\begin{align}\label{Iyr}
I_1(y,\ell)&\leq C\cdot w(Q(y,\ell))^{1/{\alpha}-1/p-1/s}\big\|f\cdot\chi_{Q(y,2\sqrt{n}\ell)}\big\|_{L^p(\nu)}\notag\\
&=C\cdot w(Q(y,2\sqrt{n}\ell))^{1/{\alpha}-1/p-1/s}\big\|f\cdot\chi_{Q(y,2\sqrt{n}\ell)}\big\|_{L^p(\nu)}\notag\\
&\times\frac{w(Q(y,\ell))^{1/{\alpha}-1/p-1/s}}{w(Q(y,2\sqrt{n}\ell))^{1/{\alpha}-1/p-1/s}}.
\end{align}
Moreover, since $1/{\alpha}-1/p-1/s<0$ and $w\in \Delta_2$, then by doubling inequality \eqref{weights}, we get
\begin{equation}\label{doubling3}
\frac{w(Q(y,\ell))^{1/{\alpha}-1/p-1/s}}{w(Q(y,2\sqrt{n}\ell))^{1/{\alpha}-1/p-1/s}}\leq C.
\end{equation}
Substituting the above inequality \eqref{doubling3} into \eqref{Iyr}, we thus obtain
\begin{align}\label{I1yr}
I_1(y,\ell)&\leq C\cdot w(Q(y,2\sqrt{n}\ell))^{1/{\alpha}-1/p-1/s}\big\|f\cdot\chi_{Q(y,2\sqrt{n}\ell)}\big\|_{L^p(\nu)}.
\end{align}
We now estimate the second term $I_2(y,\ell)$. Recall that by the definition of $I_{\gamma}$, the following estimate holds for any $x\in Q(y,\ell)$:
\begin{equation}\label{alpha1}
\big|I_{\gamma}(f_2)(x)\big|\leq C\sum_{j=1}^\infty\frac{1}{|Q(y,2^{j+1}\sqrt{n}\ell)|^{1-{\gamma}/n}}\int_{Q(y,2^{j+1}\sqrt{n}\ell)}|f(z)|\,dz.
\end{equation}
This pointwise estimate \eqref{alpha1} along with Chebyshev's inequality implies that
\begin{equation*}
\begin{split}
I_2(y,\ell)&\leq 2\cdot w(Q(y,\ell))^{1/{\beta}-1/q-1/s}\bigg(\int_{Q(y,\ell)}\big|I_{\gamma}(f_2)(x)\big|^qw(x)\,dx\bigg)^{1/q}\\
&\leq C\cdot w(Q(y,\ell))^{1/{\beta}-1/s}
\sum_{j=1}^\infty\frac{1}{|Q(y,2^{j+1}\sqrt{n}\ell)|^{1-{\gamma}/n}}\int_{Q(y,2^{j+1}\sqrt{n}\ell)}|f(z)|\,dz.
\end{split}
\end{equation*}
A further application of H\"older's inequality yields
\begin{equation*}
\begin{split}
I_2(y,\ell)&\leq C\cdot w(Q(y,\ell))^{1/{\beta}-1/s}
\sum_{j=1}^\infty\frac{1}{|Q(y,2^{j+1}\sqrt{n}\ell)|^{1-{\gamma}/n}}\bigg(\int_{Q(y,2^{j+1}\sqrt{n}\ell)}|f(z)|^p\nu(z)\,dz\bigg)^{1/p}\\
&\times\bigg(\int_{Q(y,2^{j+1}\sqrt{n}\ell)}\nu(z)^{-p'/p}\,dz\bigg)^{1/{p'}}\\
&=C\sum_{j=1}^\infty w\big(Q(y,2^{j+1}\sqrt{n}\ell)\big)^{1/{\beta}-1/q-1/s}\big\|f\cdot\chi_{Q(y,2^{j+1}\sqrt{n}\ell)}\big\|_{L^p(\nu)}\\
&\times\frac{w(Q(y,\ell))^{1/{\beta}-1/s}}{w(Q(y,2^{j+1}\sqrt{n}\ell))^{1/{\beta}-1/s}}
\cdot\frac{w(Q(y,2^{j+1}\sqrt{n}\ell))^{1/q}}{|Q(y,2^{j+1}\sqrt{n}\ell)|^{1-{\gamma}/n}}\\
&\times\bigg(\int_{Q(y,2^{j+1}\sqrt{n}\ell)}\nu(z)^{-p'/p}\,dz\bigg)^{1/{p'}}.
\end{split}
\end{equation*}
In addition, we apply H\"older's inequality again with exponent $r>1$ to get
\begin{align}\label{U2}
w\big(Q(y,2^{j+1}\sqrt{n}\ell)\big)&=\int_{Q(y,2^{j+1}\sqrt{n}\ell)}w(z)\,dz\notag\\
&\leq\big|Q(y,2^{j+1}\sqrt{n}\ell)\big|^{1/{r'}}\bigg(\int_{Q(y,2^{j+1}\sqrt{n}\ell)}w(z)^r\,dz\bigg)^{1/r}.
\end{align}
Hence, in view of \eqref{U2} and \eqref{WI}, we have
\begin{align}\label{I2yr}
I_2(y,\ell)&\leq C\sum_{j=1}^\infty w\big(Q(y,2^{j+1}\sqrt{n}\ell)\big)^{1/{\alpha}-1/p-1/s}\big\|f\cdot\chi_{Q(y,2^{j+1}\sqrt{n}\ell)}\big\|_{L^p(\nu)}
\cdot\frac{w(Q(y,\ell))^{1/{\beta}-1/s}}{w(Q(y,2^{j+1}\sqrt{n}\ell))^{1/{\beta}-1/s}}\notag\\
&\times\frac{|Q(y,2^{j+1}\sqrt{n}\ell)|^{1/{(r'q)}}}{|Q(y,2^{j+1}\sqrt{n}\ell)|^{1-{\gamma}/n}}
\bigg(\int_{Q(y,2^{j+1}\sqrt{n}\ell)}w(z)^r\,dz\bigg)^{1/{(rq)}}
\bigg(\int_{Q(y,2^{j+1}\sqrt{n}\ell)}\nu(z)^{-p'/p}\,dz\bigg)^{1/{p'}}\notag\\
&\leq C\sum_{j=1}^\infty w\big(Q(y,2^{j+1}\sqrt{n}\ell)\big)^{1/{\alpha}-1/p-1/s}\big\|f\cdot\chi_{Q(y,2^{j+1}\sqrt{n}\ell)}\big\|_{L^p(\nu)}
\cdot\frac{w(Q(y,\ell))^{1/{\beta}-1/s}}{w(Q(y,2^{j+1}\sqrt{n}\ell))^{1/{\beta}-1/s}}.
\end{align}
The last inequality is obtained by the $A_p$-type condition $(\S)$ assumed on $(w,\nu)$. Furthermore, arguing as in the proof of Theorem \ref{mainthm:1}, we know that for any positive integer $j$, there exists a \emph{reverse doubling constant }$D=D(w)>1$ independent of $Q(y,\ell)$ such that
\begin{equation*}
w\big(Q(y,2^{j+1}\sqrt{n}\ell)\big)\geq D^{j+1}\cdot w(Q(y,\ell)).
\end{equation*}
Hence, we compute
\begin{align}\label{5}
\sum_{j=1}^\infty\frac{w(Q(y,\ell))^{1/{\beta}-1/s}}{w(Q(y,2^{j+1}\sqrt{n}\ell))^{1/{\beta}-1/s}}
&\leq \sum_{j=1}^\infty\left(\frac{w(Q(y,\ell))}{D^{j+1}\cdot w(Q(y,\ell))}\right)^{1/{\beta}-1/s}\notag\\
&=\sum_{j=1}^\infty\left(\frac{1}{D^{j+1}}\right)^{1/{\beta}-1/s}\leq C,
\end{align}
where the last series is convergent since the \emph{reverse doubling constant }$D>1$ and $1/{\beta}-1/s>0$. Therefore by taking the $L^s(\mu)$-norm of both sides of \eqref{Ip}(with respect to the variable $y$), and then using Minkowski's inequality, \eqref{I1yr} and \eqref{I2yr}, we obtain
\begin{equation*}
\begin{split}
&\Big\|w(Q(y,\ell))^{1/{\beta}-1/q-1/s}\big\|I_{\gamma}(f)\cdot\chi_{Q(y,\ell)}\big\|_{WL^q(w)}\Big\|_{L^s(\mu)}\\
&\leq\big\|I_1(y,\ell)\big\|_{L^s(\mu)}+\big\|I_2(y,\ell)\big\|_{L^s(\mu)}\\
&\leq C\Big\|w\big(Q(y,2\sqrt{n}\ell)\big)^{1/{\alpha}-1/p-1/s}\big\|f\cdot\chi_{Q(y,2\sqrt{n}\ell)}\big\|_{L^p(\nu)}\Big\|_{L^s(\mu)}\\
&+C\sum_{j=1}^\infty\Big\|w\big(Q(y,2^{j+1}\sqrt{n}\ell)\big)^{1/{\alpha}-1/p-1/s}\big\|f\cdot\chi_{Q(y,2^{j+1}\sqrt{n}\ell)}\big\|_{L^p(\nu)}\Big\|_{L^s(\mu)}
\times\frac{w(Q(y,\ell))^{1/{\beta}-1/s}}{w(Q(y,2^{j+1}\sqrt{n}\ell))^{1/{\beta}-1/s}}\\
&\leq C\big\|f\big\|_{(L^p,L^s)^{\alpha}(\nu,w;\mu)}+C\big\|f\big\|_{(L^p,L^s)^{\alpha}(\nu,w;\mu)}
\times\sum_{j=1}^\infty\frac{w(Q(y,\ell))^{1/{\beta}-1/s}}{w(Q(y,2^{j+1}\sqrt{n}\ell))^{1/{\beta}-1/s}}\\
&\leq C\big\|f\big\|_{(L^p,L^s)^{\alpha}(\nu,w;\mu)},
\end{split}
\end{equation*}
where the last inequality follows from \eqref{5}. Thus, by taking the supremum over all $\ell>0$, we finish the proof of Theorem \ref{mainthm:2}.
\end{proof}

\section{Proofs of Theorems \ref{mainthm:3} and \ref{mainthm:4}}

For the results involving commutators, we need the following properties of $\mathrm{BMO}(\mathbb R^n)$, which can be found in \cite{perez1} and \cite{wang2}.
\begin{lemma}\label{BMO}
Let $b$ be a function in $\mathrm{BMO}(\mathbb R^n)$.

$(i)$ For every cube $Q$ in $\mathbb R^n$ and for any positive integer $j$, then
\begin{equation*}
\big|b_{2^{j+1}Q}-b_Q\big|\leq C\cdot(j+1)\|b\|_*.
\end{equation*}

$(ii)$ Let $1<p<\infty$. For every cube $Q$ in $\mathbb R^n$ and for any $w\in A_{\infty}$, then
\begin{equation*}
\bigg(\int_Q\big|b(x)-b_Q\big|^pw(x)\,dx\bigg)^{1/p}\leq C\|b\|_*\cdot w(Q)^{1/p}.
\end{equation*}
\end{lemma}

Before proving our main theorems, we will also need a generalization of H\"older's inequality due to O'Neil \cite{neil}.

\begin{lemma}\label{three}
Let $\mathcal A$, $\mathcal B$ and $\mathcal C$ be Young functions such that for all $t>0$,
\begin{equation*}
\mathcal A^{-1}(t)\cdot\mathcal B^{-1}(t)\leq\mathcal C^{-1}(t),
\end{equation*}
where $\mathcal A^{-1}(t)$ is the inverse function of $\mathcal A(t)$. Then for all functions $f$ and $g$ and all cubes $Q$ in $\mathbb R^n$,
\begin{equation*}
\big\|f\cdot g\big\|_{\mathcal C,Q}\leq 2\big\|f\big\|_{\mathcal A,Q}\big\|g\big\|_{\mathcal B,Q}.
\end{equation*}
\end{lemma}

We are now ready to give the proofs of Theorems \ref{mainthm:3} and \ref{mainthm:4}.

\begin{proof}[Proof of Theorem $\ref{mainthm:3}$]
Let $f\in L^{p,\kappa}(\nu,w)$ with $1<p<q<\infty$ and $0<\kappa<p/q$. For any given cube $Q=Q(x_0,\ell)\subset\mathbb R^n$, as before, we decompose $f$ as
\begin{equation*}
\begin{cases}
f=f_1+f_2\in L^{p,\kappa}(\nu,w);\  &\\
f_1=f\cdot\chi_{2Q};\  &\\
f_2=f\cdot\chi_{(2Q)^c},
\end{cases}
\end{equation*}
where $2Q:=Q(x_0,2\sqrt{n}\ell)$. Then for any given $\sigma>0$, one writes
\begin{equation*}
\begin{split}
&\frac{1}{w(Q)^{{(\kappa q)}/p\cdot 1/q}}\sigma\cdot \Big[w\big(\big\{x\in Q:\big|[b,I_\gamma](f)(x)\big|>\sigma\big\}\big)\Big]^{1/q}\\
\leq &\frac{1}{w(Q)^{\kappa/p}}\sigma\cdot \Big[w\big(\big\{x\in Q:\big|[b,I_\gamma](f_1)(x)\big|>\sigma/2\big\}\big)\Big]^{1/q}\\
&+\frac{1}{w(Q)^{\kappa/p}}\sigma\cdot \Big[w\big(\big\{x\in Q:\big|[b,I_\gamma](f_2)(x)\big|>\sigma/2\big\}\big)\Big]^{1/q}\\
:=&J_1+J_2.
\end{split}
\end{equation*}
Applying Theorem \ref{Two2} and doubling inequality \eqref{weights}, then we have
\begin{equation*}
\begin{split}
J_1&\leq C\cdot\frac{1}{w(Q)^{\kappa/p}}\left(\int_{\mathbb R^n}|f_1(x)|^p \nu(x)\,dx\right)^{1/p}\\
&=C\cdot\frac{1}{w(Q)^{\kappa/p}}\left(\int_{2Q}|f(x)|^p \nu(x)\,dx\right)^{1/p}\\
&\leq C\big\|f\big\|_{L^{p,\kappa}(\nu,w)}\cdot\frac{w(2Q)^{\kappa/p}}{w(Q)^{\kappa/p}}\\
&\leq C\big\|f\big\|_{L^{p,\kappa}(\nu,w)}.
\end{split}
\end{equation*}
On the other hand, for any $x\in Q$, from the definition \eqref{lfrac}, it then follows that
\begin{equation*}
\begin{split}
\big|[b,I_\gamma](f_2)(x)\big|&\leq C\int_{\mathbb R^n}\frac{|b(x)-b(y)|\cdot|f_2(y)|}{|x-y|^{n-\gamma}}dy\\
&\leq C\big|b(x)-b_{Q}\big|\cdot\int_{\mathbb R^n}\frac{|f_2(y)|}{|x-y|^{n-\gamma}}dy
+C\int_{\mathbb R^n}\frac{|b(y)-b_{Q}|\cdot|f_2(y)|}{|x-y|^{n-\gamma}}dy\\
&:=\xi(x)+\eta(x).
\end{split}
\end{equation*}
Thus, we can further split $J_2$ into two parts as follows:
\begin{equation*}
\begin{split}
J_2\leq&\frac{1}{w(Q)^{\kappa/p}}\sigma\cdot\Big[w\big(\big\{x\in Q:\xi(x)>\sigma/4\big\}\big)\Big]^{1/q}
+\frac{1}{w(Q)^{\kappa/p}}\sigma\cdot\Big[w\big(\big\{x\in Q:\eta(x)>\sigma/4\big\}\big)\Big]^{1/q}\\
:=&J_3+J_4.
\end{split}
\end{equation*}
Using the pointwise estimate \eqref{pointwise1} and Chebyshev's inequality, we obtain that
\begin{equation*}
\begin{split}
J_3&\leq\frac{4}{w(Q)^{\kappa/p}}\cdot\left(\int_Q\big|\xi(x)\big|^qw(x)\,dx\right)^{1/q}\\
&\leq\frac{C}{w(Q)^{\kappa/p}}\cdot\left(\int_Q \big|b(x)-b_{Q}\big|^qw(x)\,dx\right)^{1/q}
\sum_{j=1}^\infty\frac{1}{|2^{j+1}Q|^{1-\gamma/n}}\int_{2^{j+1}Q}|f(y)|\,dy\\
&\leq C\|b\|_*\cdot w(Q)^{1/q-\kappa/p}
\sum_{j=1}^\infty\frac{1}{|2^{j+1}Q|^{1-\gamma/n}}\int_{2^{j+1}Q}|f(y)|\,dy,
\end{split}
\end{equation*}
where the last inequality is due to the assumption $w\in A_{\infty}$ and Lemma \ref{BMO} $(ii)$. By the same manner as in the proof of Theorem \ref{mainthm:1}, we can also show that
\begin{equation*}
J_3\leq C\big\|f\big\|_{L^{p,\kappa}(\nu,w)}.
\end{equation*}
Similar to the proof of \eqref{pointwise1}, for all $x\in Q$, we can show the following pointwise inequality as well.
\begin{align}\label{pointwise2}
\big|\eta(x)\big|
&\leq C\int_{(2Q)^c}\frac{|b(y)-b_{Q}|\cdot|f(y)|}{|x_0-y|^{n-\gamma}}dy\notag\\
&\leq C\sum_{j=1}^\infty\frac{1}{|2^{j+1}Q|^{1-\gamma/n}}\int_{2^{j+1}Q}\big|b(y)-b_{Q}\big|\cdot\big|f(y)\big|\,dy.
\end{align}
This, together with Chebyshev's inequality, yields
\begin{equation*}
\begin{split}
J_4&\leq\frac{4}{w(Q)^{\kappa/p}}\cdot\left(\int_Q\big|\eta(x)\big|^qw(x)\,dx\right)^{1/q}\\
&\leq C\cdot w(Q)^{1/q-\kappa/p}
\sum_{j=1}^\infty\frac{1}{|2^{j+1}Q|^{1-\gamma/n}}\int_{2^{j+1}Q}\big|b(y)-b_{Q}\big|\cdot\big|f(y)\big|\,dy\\
&\leq C\cdot w(Q)^{1/q-\kappa/p}
\sum_{j=1}^\infty\frac{1}{|2^{j+1}Q|^{1-\gamma/n}}\int_{2^{j+1}Q}\big|b(y)-b_{{2^{j+1}Q}}\big|\cdot\big|f(y)\big|\,dy\\
&+C\cdot w(Q)^{1/q-\kappa/p}
\sum_{j=1}^\infty\frac{1}{|2^{j+1}Q|^{1-\gamma/n}}\int_{2^{j+1}Q}\big|b_{{2^{j+1}Q}}-b_{Q}\big|\cdot\big|f(y)\big|\,dy\\
&:=J_5+J_6.
\end{split}
\end{equation*}
An application of H\"older's inequality leads to that
\begin{equation*}
\begin{split}
J_5&\leq C\cdot w(Q)^{1/q-\kappa/p}\sum_{j=1}^\infty\frac{1}{|2^{j+1}Q|^{1-\gamma/n}}
\left(\int_{2^{j+1}Q}|f(y)|^p\nu(y)\,dy\right)^{1/p}\\
&\times\left(\int_{2^{j+1}Q}\big|b(y)-b_{{2^{j+1}Q}}\big|^{p'}\nu(y)^{-p'/p}\,dy\right)^{1/{p'}}\\
\end{split}
\end{equation*}
\begin{equation*}
\begin{split}
&\leq C\big\|f\big\|_{L^{p,\kappa}(\nu,w)}\cdot w(Q)^{1/q-\kappa/p}\sum_{j=1}^\infty\frac{w(2^{j+1}Q)^{\kappa/p}}{|2^{j+1}Q|^{1-\gamma/n}}\\
&\times\big|2^{j+1}Q\big|^{1/{p'}}\Big\|\big[b-b_{{2^{j+1}Q}}\big]\cdot \nu^{-1/p}\Big\|_{\mathcal C,2^{j+1}Q},
\end{split}
\end{equation*}
where $\mathcal C(t)=t^{p'}$ is a Young function by \eqref{norm}. For $1<p<\infty$, it is immediate that the inverse function of $\mathcal C(t)$ is $\mathcal C^{-1}(t)=t^{1/{p'}}$. Also observe that the following identity is true:
\begin{equation*}
\begin{split}
\mathcal C^{-1}(t)=t^{1/{p'}}&=\frac{t^{1/{p'}}}{1+\log^+t}\cdot\big(1+\log^+t\big)\\
&:=\mathcal A^{-1}(t)\cdot\mathcal B^{-1}(t).
\end{split}
\end{equation*}
Moreover, it is easy to see that
\begin{equation*}
\mathcal A(t)\approx  t^{p'}(1+\log^+t)^{p'}\qquad \&\qquad \mathcal B(t)\approx \exp(t)-1.
\end{equation*}
Let $\|h\|_{\exp L,Q}$ denote the mean Luxemburg norm of $h$ on cube $Q$ with Young function $\mathcal B(t)\approx \exp(t)-1$. According to Lemma \ref{three}, we thus have
\begin{align}\label{key}
\Big\|\big[b-b_{{2^{j+1}Q}}\big]\cdot \nu^{-1/p}\Big\|_{\mathcal C,2^{j+1}Q}
&\leq C\big\|b-b_{{2^{j+1}Q}}\big\|_{\exp L,2^{j+1}Q}\cdot\big\|\nu^{-1/p}\big\|_{\mathcal A,2^{j+1}Q}\notag\\
&\leq C\|b\|_*\cdot\big\|\nu^{-1/p}\big\|_{\mathcal A,2^{j+1}Q},
\end{align}
where in the last inequality we have used the well-known fact that (see \cite{perez1})
\begin{equation}\label{Jensen}
\big\|b-b_{Q}\big\|_{\exp L,Q}\leq C\|b\|_*,\qquad \mbox{for any cube }Q\subset\mathbb R^n.
\end{equation}
Indeed, the above inequality \eqref{Jensen} is equivalent to the following inequality
\begin{equation*}
\frac{1}{|Q|}\int_Q\exp\bigg(\frac{|b(y)-b_Q|}{c_0\|b\|_*}\bigg)\,dy\leq C,\qquad \mbox{for any cube }Q\subset\mathbb R^n,
\end{equation*}
which is an immediate consequence of the celebrated John--Nirenberg's inequality (see \cite{john}). Consequently, in view of \eqref{key} and \eqref{U}, we can deduce that
\begin{equation*}
\begin{split}
J_5&\leq C\big\|f\big\|_{L^{p,\kappa}(\nu,w)}
\sum_{j=1}^\infty\frac{w(Q)^{1/q-\kappa/p}}{w(2^{j+1}Q)^{1/q-\kappa/p}}
\cdot\frac{w(2^{j+1}Q)^{1/q}}{|2^{j+1}Q|^{1/p-\gamma/n}}\cdot\|b\|_*\big\|\nu^{-1/p}\big\|_{\mathcal A,2^{j+1}Q}\\
&\leq C\|b\|_*\big\|f\big\|_{L^{p,\kappa}(\nu,w)}\sum_{j=1}^\infty\frac{w(Q)^{1/q-\kappa/p}}{w(2^{j+1}Q)^{1/q-\kappa/p}}\\
&\times\big|2^{j+1}Q\big|^{\gamma/n+1/q-1/p}\cdot
\left(\frac{1}{|2^{j+1}Q|}\int_{2^{j+1}Q}w(x)^r\,dx\right)^{1/{(rq)}}\cdot\big\|\nu^{-1/p}\big\|_{\mathcal A,2^{j+1}Q}.\\
\end{split}
\end{equation*}
Since $w\in A_\infty$, we know that $w\in\Delta_2$. Furthermore, by the $A_p$-type condition $(\S\S)$ on $(w,\nu)$ and the estimate \eqref{C1}, we obtain
\begin{equation*}
\begin{split}
J_5&\leq C\big\|f\big\|_{L^{p,\kappa}(\nu,w)}\times\sum_{j=1}^\infty\frac{w(Q)^{1/q-\kappa/p}}{w(2^{j+1}Q)^{1/q-\kappa/p}}\\
&\leq C\big\|f\big\|_{L^{p,\kappa}(\nu,w)}.
\end{split}
\end{equation*}
It remains to estimate the last term $J_6$. Making use of the first part of Lemma \ref{BMO} and H\"older's inequality, we get
\begin{equation*}
\begin{split}
J_6&\leq C\cdot w(Q)^{1/q-\kappa/p}
\sum_{j=1}^\infty\frac{(j+1)\|b\|_*}{|2^{j+1}Q|^{1-\gamma/n}}\int_{2^{j+1}Q}|f(y)|\,dy\\
&\leq C\cdot w(Q)^{1/q-\kappa/p}
\sum_{j=1}^\infty\frac{(j+1)\|b\|_*}{|2^{j+1}Q|^{1-\gamma/n}}
\left(\int_{2^{j+1}Q}|f(y)|^p\nu(y)\,dy\right)^{1/p}\left(\int_{2^{j+1}Q}\nu(y)^{-p'/p}\,dy\right)^{1/{p'}}\\
&\leq C\big\|f\big\|_{L^{p,\kappa}(\nu,w)}\cdot w(Q)^{1/q-\kappa/p}
\sum_{j=1}^\infty(j+1)\cdot\frac{w(2^{j+1}Q)^{\kappa/p}}{|2^{j+1}Q|^{1-\gamma/n}}\left(\int_{2^{j+1}Q}\nu(y)^{-p'/p}\,dy\right)^{1/{p'}}\\
&=C\big\|f\big\|_{L^{p,\kappa}(\nu,w)}\sum_{j=1}^\infty(j+1)\cdot
\frac{w(Q)^{1/q-\kappa/p}}{w(2^{j+1}Q)^{1/q-\kappa/p}}\cdot\frac{w(2^{j+1}Q)^{1/q}}{|2^{j+1}Q|^{1-\gamma/n}}
\left(\int_{2^{j+1}Q}\nu(y)^{-p'/p}\,dy\right)^{1/{p'}}.
\end{split}
\end{equation*}
Let $\mathcal C(t)$ and $\mathcal A(t)$ be the same as before. Clearly, $\mathcal C(t)\leq\mathcal A(t)$ for all $t>0$, then for any cube $Q$ in $\mathbb R^n$, one has $\big\|f\big\|_{\mathcal C,Q}\leq\big\|f\big\|_{\mathcal A,Q}$ by definition, which implies that the condition $(\S\S)$ is stronger than the condition $(\S)$. This fact together with \eqref{U} yields
\begin{equation*}
\begin{split}
J_6&\leq C\big\|f\big\|_{L^{p,\kappa}(\nu,w)}\sum_{j=1}^\infty(j+1)\cdot
\frac{w(Q)^{1/q-\kappa/p}}{w(2^{j+1}Q)^{1/q-\kappa/p}}\cdot\frac{|2^{j+1}Q|^{1/{(r'q)}}}{|2^{j+1}Q|^{1-\gamma/n}}\\
&\times\left(\int_{2^{j+1}Q}w(y)^r\,dy\right)^{1/{(rq)}}\left(\int_{2^{j+1}Q}\nu(y)^{-p'/p}\,dy\right)^{1/{p'}}\\
&\leq C\big\|f\big\|_{L^{p,\kappa}(\nu,w)}
\sum_{j=1}^\infty(j+1)\cdot\frac{w(Q)^{1/q-\kappa/p}}{w(2^{j+1}Q)^{1/q-\kappa/p}}.
\end{split}
\end{equation*}
Moreover, by our hypothesis on $w:w\in A_\infty$, then there exists a number $\delta>0$ such that the inequality (\ref{compare}) holds, and hence we compute
\begin{align}\label{C2}
\sum_{j=1}^\infty(j+1)\cdot\frac{w(Q)^{1/q-\kappa/p}}{w(2^{j+1}Q)^{1/q-\kappa/p}}
&\leq C\sum_{j=1}^\infty(j+1)\cdot\left(\frac{|Q|}{|2^{j+1}Q|}\right)^{\delta(1/q-\kappa/p)}\notag\\
&=C\sum_{j=1}^\infty(j+1)\cdot\left(\frac{1}{2^{(j+1)n}}\right)^{\delta(1/q-\kappa/p)}\notag\\
&\leq C,
\end{align}
where the last series is convergent since the exponent $\delta(1/q-\kappa/p)$ is positive. This implies our desired estimate
\begin{equation*}
J_6\leq C\big\|f\big\|_{L^{p,\kappa}(\nu,w)}.
\end{equation*}
Summarizing the estimates derived above, and then taking the supremum over all cubes $Q\subset\mathbb R^n$ and all $\sigma>0$, we conclude the proof of Theorem \ref{mainthm:3}.
\end{proof}

\begin{proof}[Proof of Theorem $\ref{mainthm:4}$]
Let $1<p\leq\alpha<s\leq\infty$ and $f\in(L^p,L^s)^{\alpha}(\nu,w;\mu)$ with $w\in A_\infty$ and $\mu\in\Delta_2$. For any fixed cube $Q=Q(y,\ell)$ in $\mathbb R^n$, as usual, we decompose $f$ as
\begin{equation*}
\begin{cases}
f=f_1+f_2\in (L^p,L^s)^{\alpha}(\nu,w;\mu);\  &\\
f_1=f\cdot\chi_{2Q};\  &\\
f_2=f\cdot\chi_{(2Q)^c},
\end{cases}
\end{equation*}
where $2Q=Q(y,2\sqrt{n}\ell)$. Then for given $y\in\mathbb R^n$ and $\ell>0$, we write
\begin{align}\label{Jprime}
&w(Q(y,\ell))^{1/{\beta}-1/q-1/s}\big\|[b,I_{\gamma}](f)\cdot\chi_{Q(y,\ell)}\big\|_{WL^q(w)}\notag\\
&\leq 2\cdot w(Q(y,\ell))^{1/{\beta}-1/q-1/s}\big\|[b,I_{\gamma}](f_1)\cdot\chi_{Q(y,\ell)}\big\|_{WL^q(w)}\notag\\
&+2\cdot w(Q(y,\ell))^{1/{\beta}-1/q-1/s}\big\|[b,I_{\gamma}](f_2)\cdot\chi_{Q(y,\ell)}\big\|_{WL^q(w)}\notag\\
&:=J_1(y,\ell)+J_2(y,\ell).
\end{align}
Next we shall calculate the two terms, respectively. According to Theorem \ref{Two2}, we get
\begin{align*}
J_1(y,\ell)&\leq 2\cdot w(Q(y,\ell))^{1/{\beta}-1/q-1/s}\big\|[b,I_{\gamma}](f_1)\big\|_{WL^q(w)}\notag\\
&\leq C\cdot w(Q(y,\ell))^{1/{\beta}-1/q-1/s}
\bigg(\int_{Q(y,2\sqrt{n}\ell)}|f(x)|^p\nu(x)\,dx\bigg)^{1/p}\notag\\
&=C\cdot w\big(Q(y,2\sqrt{n}\ell)\big)^{1/{\alpha}-1/p-1/s}\big\|f\cdot\chi_{Q(y,2\sqrt{n}\ell)}\big\|_{L^p(\nu)}\notag\\
&\times \frac{w(Q(y,\ell))^{1/{\alpha}-1/p-1/s}}{w(Q(y,2\sqrt{n}\ell))^{1/{\alpha}-1/p-1/s}}\notag,
\end{align*}
where the last identity is due to \eqref{WI}. Moreover, since $w\in A_\infty$, we know that $w\in\Delta_2$, and hence by inequality (\ref{doubling3}),
\begin{equation}\label{J1prime}
J_1(y,\ell)\leq C\cdot w\big(Q(y,2\sqrt{n}\ell)\big)^{1/{\alpha}-1/p-1/s}\big\|f\cdot\chi_{Q(y,2\sqrt{n}\ell)}\big\|_{L^p(\nu)}.
\end{equation}
On the other hand, from the definition \eqref{lfrac}, one can see that for any $x\in Q(y,\ell)$,
\begin{equation*}
\begin{split}
\big|[b,I_{\gamma}](f_2)(x)\big|
&\leq \big|b(x)-b_{Q(y,\ell)}\big|\cdot\big|I_\gamma(f_2)(x)\big|
+\big|I_\gamma\big([b_{Q(y,\ell)}-b]f_2\big)(x)\big|\\
&:=\widetilde{\xi}(x)+\widetilde{\eta}(x).
\end{split}
\end{equation*}
Consequently, we can further divide $J_2(y,\ell)$ into two parts:
\begin{equation*}
\begin{split}
J_2(y,\ell)\leq&4\cdot w(Q(y,\ell))^{1/{\beta}-1/q-1/s}\big\|\widetilde{\xi}(\cdot)\cdot\chi_{Q(y,\ell)}\big\|_{WL^q(w)}\\
&+4\cdot w(Q(y,\ell))^{1/{\beta}-1/q-1/s}\big\|\widetilde{\eta}(\cdot)\cdot\chi_{Q(y,\ell)}\big\|_{WL^q(w)}\\
:=&J_3(y,\ell)+J_4(y,\ell).
\end{split}
\end{equation*}
For the term $J_3(y,\ell)$, it follows directly from Chebyshev's inequality and estimate \eqref{alpha1} that
\begin{equation*}
\begin{split}
J_3(y,\ell)&\leq4\cdot w(Q(y,\ell))^{1/{\beta}-1/q-1/s}\bigg(\int_{Q(y,\ell)}\big|\widetilde{\xi}(x)\big|^qw(x)\,dx\bigg)^{1/q}\\
&\leq C\cdot w(Q(y,\ell))^{1/{\beta}-1/q-1/s}\bigg(\int_{Q(y,\ell)}\big|b(x)-b_{Q(y,\ell)}\big|^qw(x)\,dx\bigg)^{1/q}\\
&\times\sum_{j=1}^\infty\frac{1}{|Q(y,2^{j+1}\sqrt{n}\ell)|^{1-\gamma/n}}\int_{Q(y,2^{j+1}\sqrt{n}\ell)}|f(z)|\,dz\\
&\leq C\cdot w(Q(y,\ell))^{1/{\beta}-1/s}
\sum_{j=1}^\infty\frac{1}{|Q(y,2^{j+1}\sqrt{n}\ell)|^{1-\gamma/n}}\int_{Q(y,2^{j+1}\sqrt{n}\ell)}|f(z)|\,dz,
\end{split}
\end{equation*}
where in the last inequality we have used the fact that $w\in A_{\infty}$ and Lemma \ref{BMO}$(ii)$. Arguing as in the proof of Theorem \ref{mainthm:2}, we can also obtain that
\begin{equation*}
J_3(y,\ell)\leq C\sum_{j=1}^\infty w\big(Q(y,2^{j+1}\sqrt{n}\ell)\big)^{1/{\alpha}-1/p-1/s}\big\|f\cdot\chi_{Q(y,2^{j+1}\sqrt{n}\ell)}\big\|_{L^p(\nu)}
\cdot\frac{w(Q(y,\ell))^{1/{\beta}-1/s}}{w(Q(y,2^{j+1}\sqrt{n}\ell))^{1/{\beta}-1/s}}.
\end{equation*}
Let us now estimate the other term $J_4(y,\ell)$. As it was shown in Theorem \ref{mainthm:3} (see \eqref{pointwise2}), the following pointwise estimate
\begin{equation*}
\begin{split}
\widetilde{\eta}(x)&=\big|I_\gamma\big([b_{Q(y,\ell)}-b]f_2\big)(x)\big|\\
&\leq C\sum_{j=1}^\infty\frac{1}{|Q(y,2^{j+1}\sqrt{n}\ell)|^{1-\gamma/n}}\int_{Q(y,2^{j+1}\sqrt{n}\ell)}\big|b(z)-b_{Q(y,\ell)}\big|\cdot|f(z)|\,dz
\end{split}
\end{equation*}
holds for any $x\in Q(y,\ell)$ by a routine argument. This, together with Chebyshev's inequality, implies that
\begin{equation*}
\begin{split}
J_4(y,\ell)&\leq4\cdot w(Q(y,\ell))^{1/{\beta}-1/q-1/s}\bigg(\int_{Q(y,\ell)}\big|\widetilde{\eta}(x)\big|^qw(x)\,dx\bigg)^{1/q}\\
&\leq C\cdot w(Q(y,\ell))^{1/{\beta}-1/s}
\sum_{j=1}^\infty\frac{1}{|Q(y,2^{j+1}\sqrt{n}\ell)|^{1-\gamma/n}}\int_{Q(y,2^{j+1}\sqrt{n}\ell)}\big|b(z)-b_{Q(y,\ell)}\big|\cdot|f(z)|\,dz\\
&\leq C\cdot w(Q(y,\ell))^{1/{\beta}-1/s}
\sum_{j=1}^\infty\frac{1}{|Q(y,2^{j+1}\sqrt{n}\ell)|^{1-\gamma/n}}\int_{Q(y,2^{j+1}\sqrt{n}\ell)}\big|b(z)-b_{Q(y,2^{j+1}\sqrt{n}\ell)}\big|\cdot|f(z)|\,dz\\
&+C\cdot w(Q(y,\ell))^{1/{\beta}-1/s}
\sum_{j=1}^\infty\frac{1}{|Q(y,2^{j+1}\sqrt{n}\ell)|^{1-\gamma/n}}\int_{Q(y,2^{j+1}\sqrt{n}\ell)}\big|b_{Q(y,2^{j+1}\sqrt{n}\ell)}-b_{Q(y,\ell)}\big|\cdot|f(z)|\,dz\\
&:=J_5(y,\ell)+J_6(y,\ell).
\end{split}
\end{equation*}
An application of H\"older's inequality leads to that
\begin{equation*}
\begin{split}
J_5(y,\ell)&\leq C\cdot w(Q(y,\ell))^{1/{\beta}-1/s}
\sum_{j=1}^\infty\frac{1}{|Q(y,2^{j+1}\sqrt{n}\ell)|^{1-\gamma/n}}\bigg(\int_{Q(y,2^{j+1}\sqrt{n}\ell)}|f(z)|^p\nu(z)\,dz\bigg)^{1/p}\\
&\times\bigg(\int_{Q(y,2^{j+1}\sqrt{n}\ell)}\big|b(z)-b_{Q(y,2^{j+1}\sqrt{n}\ell)}\big|^{p'}\nu(z)^{-p'/p}\,dz\bigg)^{1/{p'}}\\
&=C\cdot w(Q(y,\ell))^{1/{\beta}-1/s}
\sum_{j=1}^\infty\frac{1}{|Q(y,2^{j+1}\sqrt{n}\ell)|^{1-\gamma/n}}\big\|f\cdot\chi_{Q(y,2^{j+1}\sqrt{n}\ell)}\big\|_{L^p(\nu)}\\
&\times\big|Q(y,2^{j+1}\sqrt{n}\ell)\big|^{1/{p'}}\Big\|\big[b-b_{Q(y,2^{j+1}\sqrt{n}\ell)}\big]\cdot\nu^{-1/p}\Big\|_{\mathcal C,Q(y,2^{j+1}\sqrt{n}\ell)},
\end{split}
\end{equation*}
where $\mathcal C(t)=t^{p'}$ is a Young function. Recall that the following inequalities
\begin{align}\label{final}
&\Big\|\big[b-b_{Q(y,2^{j+1}\sqrt{n}\ell)}\big]\cdot\nu^{-1/p}\Big\|_{\mathcal C,Q(y,2^{j+1}\sqrt{n}\ell)}\notag\\
&\leq C\big\|b-b_{Q(y,2^{j+1}\sqrt{n}\ell)}\big\|_{\mathcal B,Q(y,2^{j+1}\sqrt{n}\ell)}\cdot\big\|\nu^{-1/p}\big\|_{\mathcal A,Q(y,2^{j+1}\sqrt{n}\ell)}\notag\\
&\leq C\|b\|_*\cdot\big\|\nu^{-1/p}\big\|_{\mathcal A,Q(y,2^{j+1}\sqrt{n}\ell)}
\end{align}
hold by generalized H\"older's inequality and the estimate \eqref{Jensen}, where
\begin{equation*}
\mathcal A(t)\approx t^{p'}(1+\log^+t)^{p'}\qquad \&\qquad \mathcal B(t)\approx \exp(t)-1.
\end{equation*}
Moreover, in view of \eqref{U2} and \eqref{final}, we can deduce that
\begin{equation*}
\begin{split}
J_5(y,\ell)&\leq C\|b\|_*\cdot w(Q(y,\ell))^{1/{\beta}-1/s}
\sum_{j=1}^\infty\frac{\big\|f\cdot\chi_{Q(y,2^{j+1}\sqrt{n}\ell)}\big\|_{L^p(\nu)}}{|Q(y,2^{j+1}\sqrt{n}\ell)|^{1/p-\gamma/n}}
\cdot\big\|\nu^{-1/p}\big\|_{\mathcal A,Q(y,2^{j+1}\sqrt{n}\ell)}\\
&=C\|b\|_*\sum_{j=1}^\infty w\big(Q(y,2^{j+1}\sqrt{n}\ell)\big)^{1/{\beta}-1/q-1/s}\big\|f\cdot\chi_{Q(y,2^{j+1}\sqrt{n}\ell)}\big\|_{L^p(\nu)}
\cdot\frac{w(Q(y,\ell))^{1/{\beta}-1/s}}{w(Q(y,2^{j+1}\sqrt{n}\ell))^{1/{\beta}-1/s}}\\
&\times\frac{w(Q(y,2^{j+1}\sqrt{n}\ell))^{1/q}}{|Q(y,2^{j+1}\sqrt{n}\ell)|^{1/p-\gamma/n}}
\cdot\big\|\nu^{-1/p}\big\|_{\mathcal A,Q(y,2^{j+1}\sqrt{n}\ell)}\\
&\leq C\|b\|_*\sum_{j=1}^\infty w\big(Q(y,2^{j+1}\sqrt{n}\ell)\big)^{1/{\beta}-1/q-1/s}\big\|f\cdot\chi_{Q(y,2^{j+1}\sqrt{n}\ell)}\big\|_{L^p(\nu)}
\cdot\frac{w(Q(y,\ell))^{1/{\beta}-1/s}}{w(Q(y,2^{j+1}\sqrt{n}\ell))^{1/{\beta}-1/s}}\\
&\big|Q(y,2^{j+1}\sqrt{n}\ell)\big|^{\gamma/n+1/q-1/p}\cdot
\bigg(\frac{1}{Q(y,2^{j+1}\sqrt{n}\ell)}\int_{Q(y,2^{j+1}\sqrt{n}\ell)}w(z)^r\,dz\bigg)^{1/{(rq)}}
\big\|\nu^{-1/p}\big\|_{\mathcal A,Q(y,2^{j+1}\sqrt{n}\ell)}\\
&\leq C\|b\|_*\sum_{j=1}^\infty w\big(Q(y,2^{j+1}\sqrt{n}\ell)\big)^{1/{\beta}-1/q-1/s}\big\|f\cdot\chi_{Q(y,2^{j+1}\sqrt{n}\ell)}\big\|_{L^p(\nu)}
\cdot\frac{w(Q(y,\ell))^{1/{\beta}-1/s}}{w(Q(y,2^{j+1}\sqrt{n}\ell))^{1/{\beta}-1/s}}.
\end{split}
\end{equation*}
The last inequality is obtained by the $A_p$-type condition $(\S\S)$ assumed on $(w,\nu)$. We now turn our attention to the last term $J_6(y,\ell)$. Applying Lemma \ref{BMO}$(i)$ and H\"older's inequality, we get
\begin{equation*}
\begin{split}
J_6(y,\ell)&\leq C\cdot w(Q(y,\ell))^{1/{\beta}-1/s}
\sum_{j=1}^\infty\frac{(j+1)\|b\|_*}{|Q(y,2^{j+1}\sqrt{n}\ell)|^{1-\gamma/n}}\int_{Q(y,2^{j+1}\sqrt{n}\ell)}|f(z)|\,dz\\
&\leq C\cdot w(Q(y,\ell))^{1/{\beta}-1/s}
\sum_{j=1}^\infty\frac{(j+1)\|b\|_*}{|Q(y,2^{j+1}\sqrt{n}\ell)|^{1-\gamma/n}}\bigg(\int_{Q(y,2^{j+1}\sqrt{n}\ell)}|f(z)|^p\nu(z)\,dz\bigg)^{1/p}\\
&\times\bigg(\int_{Q(y,2^{j+1}\sqrt{n}\ell)}\nu(z)^{-p'/p}\,dz\bigg)^{1/{p'}}\\
&=C\|b\|_*\sum_{j=1}^\infty w\big(Q(y,2^{j+1}\sqrt{n}\ell)\big)^{1/{\beta}-1/q-1/s}\big\|f\cdot\chi_{Q(y,2^{j+1}\sqrt{n}\ell)}\big\|_{L^p(\nu)}\\
&\times\big(j+1\big)\cdot\frac{w(Q(y,\ell))^{1/{\beta}-1/s}}{w(Q(y,2^{j+1}\sqrt{n}\ell))^{1/{\beta}-1/s}}
\cdot\frac{w(Q(y,2^{j+1}\sqrt{n}\ell))^{1/q}}{|Q(y,2^{j+1}\sqrt{n}\ell)|^{1-\gamma/n}}\bigg(\int_{Q(y,2^{j+1}\sqrt{n}\ell)}\nu(z)^{-p'/p}\,dz\bigg)^{1/{p'}}.
\end{split}
\end{equation*}
Also observe that the condition $(\S\S)$ is stronger than the condition $(\S)$. Using this fact along with \eqref{U2}, we have
\begin{equation*}
\begin{split}
J_6(y,\ell)&\leq C\|b\|_*\sum_{j=1}^\infty w\big(Q(y,2^{j+1}\sqrt{n}\ell)\big)^{1/{\beta}-1/q-1/s}\big\|f\cdot\chi_{Q(y,2^{j+1}\sqrt{n}\ell)}\big\|_{L^p(\nu)}\\
&\times\big(j+1\big)\cdot\frac{w(Q(y,\ell))^{1/{\beta}-1/s}}{w(Q(y,2^{j+1}\sqrt{n}\ell))^{1/{\beta}-1/s}}\\
&\times\frac{|Q(y,2^{j+1}\sqrt{n}\ell)|^{1/{(r'q)}}}{|Q(y,2^{j+1}\sqrt{n}\ell)|^{1-\gamma/n}}
\bigg(\int_{Q(y,2^{j+1}\sqrt{n}\ell)}w(z)^r\,dz\bigg)^{1/{(rq)}}
\bigg(\int_{Q(y,2^{j+1}\sqrt{n}\ell)}\nu(z)^{-p'/p}\,dz\bigg)^{1/{p'}}\\
&\leq C\|b\|_*\sum_{j=1}^\infty w\big(Q(y,2^{j+1}\sqrt{n}\ell)\big)^{1/{\beta}-1/q-1/s}\big\|f\cdot\chi_{Q(y,2^{j+1}\sqrt{n}\ell)}\big\|_{L^p(\nu)}\\
&\times\big(j+1\big)\cdot\frac{w(Q(y,\ell))^{1/{\beta}-1/s}}{w(Q(y,2^{j+1}\sqrt{n}\ell))^{1/{\beta}-1/s}}.
\end{split}
\end{equation*}
Summing up all the above estimates and taking into consideration \eqref{WI}, we conclude that
\begin{align}\label{J2prime}
J_2(y,\ell)&\leq C\sum_{j=1}^\infty w\big(Q(y,2^{j+1}\sqrt{n}\ell)\big)^{1/{\beta}-1/q-1/s}\big\|f\cdot\chi_{Q(y,2^{j+1}\sqrt{n}\ell)}\big\|_{L^p(\nu)}\notag\\
&\times\big(j+1\big)\cdot\frac{w(Q(y,\ell))^{1/{\beta}-1/s}}{w(Q(y,2^{j+1}\sqrt{n}\ell))^{1/{\beta}-1/s}}\notag\\
&=C\sum_{j=1}^\infty w\big(Q(y,2^{j+1}\sqrt{n}\ell)\big)^{1/{\alpha}-1/p-1/s}\big\|f\cdot\chi_{Q(y,2^{j+1}\sqrt{n}\ell)}\big\|_{L^p(\nu)}\notag\\
&\times\big(j+1\big)\cdot\frac{w(Q(y,\ell))^{1/{\beta}-1/s}}{w(Q(y,2^{j+1}\sqrt{n}\ell))^{1/{\beta}-1/s}}.
\end{align}
Moreover, by our hypothesis on $w:w\in A_\infty$ and inequality \eqref{compare} with exponent $\delta^\ast>0$, we compute
\begin{align}\label{6}
\sum_{j=1}^\infty\big(j+1\big)\cdot\frac{w(Q(y,\ell))^{1/{\beta}-1/s}}{w(Q(y,2^{j+1}\sqrt{n}\ell))^{1/{\beta}-1/s}}
&\leq C\sum_{j=1}^\infty(j+1)\cdot\left(\frac{|Q(y,\ell)|}{|Q(y,2^{j+1}\sqrt{n}\ell)|}\right)^{\delta^\ast(1/{\beta}-1/s)}\notag\\
&=C\sum_{j=1}^\infty(j+1)\cdot\left(\frac{1}{2^{(j+1)n}}\right)^{\delta^\ast(1/{\beta}-1/s)}\notag\\
&\leq C.
\end{align}
Notice that the exponent $\delta^\ast(1/{\beta}-1/s)$ is positive because $\beta<s$, which guarantees that the last series is convergent.
Thus, by taking the $L^s(\mu)$-norm of both sides of \eqref{Jprime}(with respect to the variable $y$), and then using Minkowski's inequality, \eqref{J1prime} and \eqref{J2prime}, we finally obtain
\begin{equation*}
\begin{split}
&\Big\|w(Q(y,\ell))^{1/{\beta}-1/q-1/s}\big\|[b,I_{\gamma}](f)\cdot\chi_{Q(y,\ell)}\big\|_{WL^q(w)}\Big\|_{L^s(\mu)}\\
&\leq\big\|J_1(y,\ell)\big\|_{L^s(\mu)}+\big\|J_2(y,\ell)\big\|_{L^s(\mu)}\\
&\leq C\Big\|w\big(Q(y,2\sqrt{n}\ell)\big)^{1/{\alpha}-1/p-1/s}\big\|f\cdot\chi_{Q(y,2\sqrt{n}\ell)}\big\|_{L^p(\nu)}\Big\|_{L^s(\mu)}\\
&+C\sum_{j=1}^\infty\Big\|w\big(Q(y,2^{j+1}\sqrt{n}\ell)\big)^{1/{\alpha}-1/p-1/s}\big\|f\cdot\chi_{Q(y,2^{j+1}\sqrt{n}\ell)}\big\|_{L^p(\nu)}\Big\|_{L^s(\mu)}\\
&\times\big(j+1\big)\cdot\frac{w(Q(y,\ell))^{1/{\beta}-1/s}}{w(Q(y,2^{j+1}\sqrt{n}\ell))^{1/{\beta}-1/s}}\\
&\leq C\big\|f\big\|_{(L^p,L^s)^{\alpha}(\nu,w;\mu)}+C\big\|f\big\|_{(L^p,L^s)^{\alpha}(\nu,w;\mu)}
\times\sum_{j=1}^\infty\big(j+1\big)\cdot\frac{w(Q(y,\ell))^{1/{\beta}-1/s}}{w(Q(y,2^{j+1}\sqrt{n}\ell))^{1/{\beta}-1/s}}\\
&\leq C\big\|f\big\|_{(L^p,L^s)^{\alpha}(\nu,w;\mu)},
\end{split}
\end{equation*}
where the last inequality is due to \eqref{6}. We therefore conclude the proof of Theorem \ref{mainthm:4} by taking the supremum over all $\ell>0$.
\end{proof}
The higher order commutators formed by $I_{\gamma}$ and a symbol function $b$ are usually defined by
\begin{align}\label{lfracm}
[b,I_{\gamma}]_mf(x)&:=\frac{1}{\zeta(\gamma)}\int_{\mathbb R^n}\frac{[b(x)-b(y)]^m\cdot f(y)}{|x-y|^{n-\gamma}}\,dy,\quad m=1,2,3,\dots.
\end{align}
Obviously, $[b,I_{\gamma}]_1=[b,I_{\gamma}]$ which is the linear commutator \eqref{lfrac}, and
\begin{equation*}
[b,I_{\gamma}]_m=\big[b,[b,I_{\gamma}]_{m-1}\big],\quad m=2,3,\dots.
\end{equation*}
By induction argument, we will then obtain the following conclusions.
\begin{theorem}\label{mainthm:7}
Let $0<\gamma<n$, $1<p<q<\infty$, $0<\kappa<p/q$ and $b\in\mathrm{BMO}(\mathbb R^n)$. Given a pair of weights $(w,\nu)$ on $\mathbb R^n$, suppose that for some $r>1$ and for all cubes $Q$ in $\mathbb R^n$,
\begin{equation*}
\big|Q\big|^{\gamma/n+1/q-1/p}\cdot\bigg(\frac{1}{|Q|}\int_Q w(x)^r\,dx\bigg)^{1/{(rq)}}\big\|\nu^{-1/p}\big\|_{\mathcal A_m,Q}\leq C<\infty,
\end{equation*}
where $\mathcal A_m(t)=t^{p'}(1+\log^+t)^{mp'}$, $m=2,3,\dots$. If $w\in A_\infty$, then the higher order commutators $[b,I_\gamma]_m$ are bounded from $L^{p,\kappa}(\nu,w)$ into $WL^{q,{(\kappa q)}/p}(w)$.
\end{theorem}

\begin{theorem}\label{mainthm:8}
Let $0<\gamma<n$, $1<p<q<\infty$, $\mu\in\Delta_2$ and $b\in\mathrm{BMO}(\mathbb R^n)$. Given a pair of weights $(w,\nu)$ on $\mathbb R^n$, assume that for some $r>1$ and for all cubes $Q$ in $\mathbb R^n$,
\begin{equation*}
\big|Q\big|^{\gamma/n+1/q-1/p}\cdot\bigg(\frac{1}{|Q|}\int_Q w(x)^r\,dx\bigg)^{1/{(rq)}}\big\|\nu^{-1/p}\big\|_{\mathcal A_m,Q}\leq C<\infty,
\end{equation*}
where $\mathcal A_m(t)=t^{p'}(1+\log^+t)^{mp'}$, $m=2,3,\dots$. If $p\leq\alpha<\beta<s\leq\infty$ and $w\in A_\infty$, then the higher order commutators $[b,I_{\gamma}]_m$ are bounded from $(L^p,L^s)^{\alpha}(\nu,w;\mu)$ into $(WL^q,L^s)^{\beta}(w;\mu)$ with $1/{\beta}=1/{\alpha}-(1/p-1/q)$.
\end{theorem}

\section{Proofs of Theorems \ref{mainthm:5} and \ref{mainthm:6}}
In the last section, we will prove the conclusions of Theorems \ref{mainthm:5} and \ref{mainthm:6}.
\begin{proof}[Proof of Theorem $\ref{mainthm:5}$]
Let $f\in L^{p,\kappa}(\nu,w)$ with $1<p<q<\infty$ and $\kappa=p/q$. For any given cube $Q=Q(x_0,\ell)$ in $\mathbb R^n$, it suffices to prove that the following inequality
\begin{equation}\label{end1.1}
\frac{1}{|Q|}\int_Q\big|I_{\gamma}f(x)-(I_{\gamma}f)_Q\big|\,dx\leq C\big\|f\big\|_{L^{p,\kappa}(\nu,w)}
\end{equation}
holds. Decompose $f$ as $f=f_1+f_2$, where $f_1=f\cdot\chi_{4Q}$, $f_2=f\cdot\chi_{(4Q)^c}$, $4Q=Q(x_0,4\sqrt{n}\ell)$. By the linearity of the fractional integral operator $I_{\gamma}$, the left-hand side of \eqref{end1.1} can be divided into two parts. That is,
\begin{equation*}
\begin{split}
&\frac{1}{|Q|}\int_{Q}\big|I_{\gamma}f(x)-(I_{\gamma}f)_Q\big|\,dx\\
&\leq\frac{1}{|Q|}\int_Q\big|I_{\gamma}f_1(x)-(I_{\gamma}f_1)_Q\big|\,dx+\frac{1}{|Q|}\int_Q\big|I_{\gamma}f_2(x)-(I_{\gamma}f_2)_Q\big|\,dx\\
&:=I+II.
\end{split}
\end{equation*}
For the first term $I$, it follows directly from Fubini's theorem that
\begin{equation*}
\begin{split}
I&\leq\frac{2}{|Q|}\int_Q\big|I_{\gamma}f_1(x)\big|\,dx\\
&\leq\frac{C}{|Q|}\int_Q\bigg(\int_{4Q}\frac{1}{|x-y|^{n-\gamma}}|f(y)|\,dy\bigg)dx\\
&=\frac{C}{|Q|}\int_{4Q}\bigg(\int_Q\frac{1}{|x-y|^{n-\gamma}}\,dx\bigg)|f(y)|\,dy.
\end{split}
\end{equation*}
It is clear that
\begin{equation*}
|x-y|\leq|x-x_0|+|y-x_0|\leq\frac{5n}{2}\ell
\end{equation*}
when $x\in Q$ and $y\in 4Q$. Using the transform $x-y\mapsto z$ and polar coordinates, one has
\begin{align}\label{WHI}
\int_Q\frac{1}{|x-y|^{n-\gamma}}\,dx&\leq\int_{|z|\leq\frac{5n}{2}\ell}\frac{1}{|z|^{n-\gamma}}dz\notag\\
&=w_{n-1}\cdot\int_0^{\frac{5n}{2}\ell}\frac{1}{\varrho^{n-\gamma}}\varrho^{n-1}d\varrho\notag\\
&=w_{n-1}\cdot\frac{\,1\,}{\gamma}\Big(\frac{5n}{2}\ell\Big)^{\gamma}.
\end{align}
Here we use $w_{n-1}$ to denote the measure of the unit sphere in $\mathbb R^n$. This indicates that
\begin{equation}\label{end1.2}
I\leq\frac{C}{|Q|^{1-{\gamma}/n}}\int_{4Q}|f(y)|\,dy.
\end{equation}
Using H\"older's inequality and noting the fact that $\kappa=p/q$, we have
\begin{align}\label{end1.3}
I&\leq \frac{C}{|Q|^{1-{\gamma}/n}}\bigg(\int_{4Q}|f(y)|^p\nu(y)\,dy\bigg)^{1/p}
\bigg(\int_{4Q}\nu(y)^{-p'/p}\,dy\bigg)^{1/{p'}}\notag\\
&\leq C\big\|f\big\|_{L^{p,\kappa}(\nu,w)}\cdot\frac{w(4Q)^{{\kappa}/p}}{|Q|^{1-{\gamma}/n}}
\bigg(\int_{4Q}\nu(y)^{-p'/p}\,dy\bigg)^{1/{p'}}\notag\\
&=C\big\|f\big\|_{L^{p,\kappa}(\nu,w)}\cdot\frac{w(4Q)^{1/q}}{|Q|^{1-{\gamma}/n}}\bigg(\int_{4Q}\nu(y)^{-p'/p}\,dy\bigg)^{1/{p'}}.
\end{align}
Moreover, it follows from the condition $(\S)$ on $(w,\nu)$ and \eqref{U}(consider $4Q$ instead of $2^{j+1}Q$) that
\begin{equation*}
\begin{split}
I&\leq C\big\|f\big\|_{L^{p,\kappa}(\nu,w)}\cdot\frac{|4Q|^{1/{(r'q)}}}{|4Q|^{1-{\gamma}/n}}
\bigg(\int_{4Q}w(y)^r\,dy\bigg)^{1/{(rq)}}\bigg(\int_{4Q}\nu(y)^{-p'/p}\,dy\bigg)^{1/{p'}}\\
&\leq C\big\|f\big\|_{L^{p,\kappa}(\nu,w)}.
\end{split}
\end{equation*}
For the second term $II$, by the definition \eqref{frac}, we have that for any $x\in Q$,
\begin{equation*}
\begin{split}
\big|I_{\gamma}f_2(x)-(I_{\gamma}f_2)_Q\big|&=\bigg|\frac{1}{|Q|}\int_Q\big[I_{\gamma}f_2(x)-I_{\gamma}f_2(y)\big]\,dy\bigg|\\
&=\bigg|\frac{C}{|Q|}\int_Q\bigg\{\int_{(4Q)^c}\bigg[\frac{1}{|x-z|^{n-\gamma}}-\frac{1}{|y-z|^{n-\gamma}}\bigg]f(z)\,dz\bigg\}dy\bigg|\\
&\leq\frac{C}{|Q|}\int_Q\bigg\{\int_{(4Q)^c}\bigg|\frac{1}{|x-z|^{n-\gamma}}-\frac{1}{|y-z|^{n-\gamma}}\bigg|\cdot|f(z)|\,dz\bigg\}dy.
\end{split}
\end{equation*}
Since both $x$ and $y$ are in $Q$, $z\in(4Q)^c$, by a routine geometric observation, we must have $|x-z|\geq 2|x-y|$ and $|x-z|\approx|z-x_0|$. This fact along with the mean value theorem yields
\begin{align}\label{average}
\big|I_{\gamma}f_2(x)-(I_{\gamma}f_2)_Q\big|
&\leq\frac{C}{|Q|}\int_{Q}\bigg\{\int_{(4Q)^c}\frac{|x-y|}{|x-z|^{n-\gamma+1}}\cdot|f(z)|\,dz\bigg\}dy\notag\\
&\leq\frac{C}{|Q|}\int_{Q}\bigg\{\int_{(4Q)^c}\frac{\ell}{|z-x_0|^{n-\gamma+1}}\cdot|f(z)|\,dz\bigg\}dy\notag\\
&\leq C\int_{(4Q)^c}\frac{\ell}{|z-x_0|^{n-\gamma+1}}\cdot|f(z)|\,dz\notag\\
&\leq C\sum_{j=2}^\infty\frac{1}{2^j}\cdot\frac{1}{|2^{j+1}Q|^{1-{\gamma}/n}}\int_{2^{j+1}Q}|f(z)|\,dz.
\end{align}
Another application of H\"older's inequality gives that
\begin{align*}
\big|I_{\gamma}f_2(x)-(I_{\gamma}f_2)_Q\big|&\leq C\sum_{j=2}^\infty\frac{1}{2^j}\cdot\frac{1}{|2^{j+1}Q|^{1-{\gamma}/n}}\notag\\
&\times\bigg(\int_{2^{j+1}Q}|f(z)|^p\nu(z)\,dz\bigg)^{1/p}
\bigg(\int_{2^{j+1}Q}\nu(z)^{-p'/p}\,dz\bigg)^{1/{p'}}\notag\\
&\leq C\big\|f\big\|_{L^{p,\kappa}(\nu,w)}
\sum_{j=2}^\infty\frac{1}{2^j}\cdot\frac{w(2^{j+1}Q)^{{\kappa}/p}}{|2^{j+1}Q|^{1-{\gamma}/n}}\bigg(\int_{2^{j+1}Q}\nu(z)^{-p'/p}\,dz\bigg)^{1/{p'}}\notag\\
&=C\big\|f\big\|_{L^{p,\kappa}(\nu,w)}
\sum_{j=2}^\infty\frac{1}{2^j}\cdot\frac{w(2^{j+1}Q)^{1/q}}{|2^{j+1}Q|^{1-{\gamma}/n}}\bigg(\int_{2^{j+1}Q}\nu(z)^{-p'/p}\,dz\bigg)^{1/{p'}},\notag
\end{align*}
where the last equality is due to the fact that $\kappa=p/q$. Moreover, we apply the estimate \eqref{U} and the condition $(\S)$ to get
\begin{align}\label{end1.3}
\big|I_{\gamma}f_2(x)-(I_{\gamma}f_2)_Q\big|
&\leq C\big\|f\big\|_{L^{p,\kappa}(\nu,w)}\sum_{j=2}^\infty\frac{1}{2^j}
\cdot\frac{|2^{j+1}Q|^{1/{(r'q)}}}{|2^{j+1}Q|^{1-{\gamma}/n}}\notag\\
&\times\bigg(\int_{2^{j+1}Q}w(z)^r\,dz\bigg)^{1/{(rq)}}\bigg(\int_{2^{j+1}Q}\nu(z)^{-p'/p}\,dz\bigg)^{1/{p'}}\notag\\
&\leq C\big\|f\big\|_{L^{p,\kappa}(\nu,w)}\times\sum_{j=2}^\infty\frac{1}{2^j}\notag\\
&\leq C\big\|f\big\|_{L^{p,\kappa}(\nu,w)}.
\end{align}
From the pointwise estimate \eqref{end1.3}, it readily follows that
\begin{equation*}
II=\frac{1}{|Q|}\int_Q\big|I_{\gamma}f_2(x)-(I_{\gamma}f_2)_Q\big|\,dx\leq C\big\|f\big\|_{L^{p,\kappa}(\nu,w)}.
\end{equation*}
By combining the above estimates for $I$ and $II$, we are done.
\end{proof}

\begin{proof}[Proof of Theorem $\ref{mainthm:6}$]
Let $1<p\leq\alpha<s\leq\infty$ and $f\in(L^p,L^s)^{\alpha}(\nu,w;\mu)$ with $w\in\Delta_2$ and $\mu\in\Delta_2$. For any fixed cube $Q=Q(y,\ell)$ in $\mathbb R^n$, we are going to estimate the following expression:
\begin{equation}\label{end2.1}
\frac{1}{|Q(y,\ell)|}\int_{Q(y,\ell)}\big|I_{\gamma}f(x)-(I_{\gamma}f)_{Q(y,\ell)}\big|\,dx.
\end{equation}
As usual, we decompose $f$ as $f=f_1+f_2$, where $f_1=f\cdot\chi_{4Q}$, $f_2=f\cdot\chi_{(4Q)^c}$, $4Q=Q(y,4\sqrt{n}\ell)$. By the linearity of the fractional integral operator $I_{\gamma}$, the above expression \eqref{end2.1} can be divided into two parts. That is,
\begin{equation*}
\begin{split}
&\frac{1}{|Q(y,\ell)|}\int_{Q(y,\ell)}\big|I_{\gamma}f(x)-(I_{\gamma}f)_{Q(y,\ell)}\big|\,dx\\
&\leq \frac{1}{|Q(y,\ell)|}\int_{Q(y,\ell)}\big|I_{\gamma}f_1(x)-(I_{\gamma}f_1)_{Q(y,\ell)}\big|\,dx
+\frac{1}{|Q(y,\ell)|}\int_{Q(y,\ell)}\big|I_{\gamma}f_2(x)-(I_{\gamma}f_2)_{Q(y,\ell)}\big|\,dx\\
&:=I(y,\ell)+II(y,\ell).
\end{split}
\end{equation*}
Let us first deal with the term $I(y,\ell)$. Fubini's theorem allows us to obtain
\begin{equation*}
\begin{split}
I(y,\ell)&\leq\frac{2}{|Q(y,\ell)|}\int_{Q(y,\ell)}\big|I_{\gamma}f_1(x)\big|\,dx\\
&\leq\frac{C}{|Q(y,\ell)|}\int_{Q(y,\ell)}\bigg(\int_{Q(y,4\sqrt{n}\ell)}\frac{1}{|x-z|^{n-\gamma}}|f(z)|\,dz\bigg)dx\\
&=\frac{C}{|Q(y,\ell)|}\int_{Q(y,4\sqrt{n}\ell)}\bigg(\int_{Q(y,\ell)}\frac{1}{|x-z|^{n-\gamma}}\,dx\bigg)|f(z)|\,dz\\
&\leq\frac{C}{|Q(y,\ell)|^{1-{\gamma}/n}}\int_{Q(y,4\sqrt{n}\ell)}|f(z)|\,dz,
\end{split}
\end{equation*}
where we have invoked \eqref{WHI} in the last inequality. Moreover, by H\"older's inequality, we can see that
\begin{align}\label{end2.2}
I(y,\ell)&\leq\frac{C}{|Q(y,\ell)|^{1-{\gamma}/n}}\bigg(\int_{Q(y,4\sqrt{n}\ell)}|f(z)|^p\nu(z)\,dz\bigg)^{1/p}
\bigg(\int_{Q(y,4\sqrt{n}\ell)}\nu(z)^{-p'/p}\,dz\bigg)^{1/{p'}}\notag\\
&=C\cdot w\big(Q(y,4\sqrt{n}\ell)\big)^{1/{\alpha}-1/p-1/s}\big\|f\cdot\chi_{Q(y,4\sqrt{n}\ell)}\big\|_{L^p(\nu)}\notag\\
&\times\frac{w(Q(y,4\sqrt{n}\ell))^{1/q}}{|Q(y,\ell)|^{1-{\gamma}/n}}
\bigg(\int_{Q(y,4\sqrt{n}\ell)}\nu(z)^{-p'/p}\,dz\bigg)^{1/{p'}},
\end{align}
where in the last equality we have used the hypothesis $1/s=1/{\alpha}-(1/p-1/q)$. Taking into consideration \eqref{U2} and the condition $(\S)$ on $(w,\nu)$, we further obtain that
\begin{align}\label{IW}
&I(y,\ell)\leq C\cdot w\big(Q(y,4\sqrt{n}\ell)\big)^{1/{\alpha}-1/p-1/s}\big\|f\cdot\chi_{Q(y,4\sqrt{n}\ell)}\big\|_{L^p(\nu)}\notag\\
&\times\frac{|Q(y,4\sqrt{n}\ell)|^{1/{(r'q)}}}{|Q(y,4\sqrt{n}\ell)|^{1-{\gamma}/n}}
\bigg(\int_{Q(y,4\sqrt{n}\ell)}w(z)^r\,dz\bigg)^{1/{(rq)}}
\bigg(\int_{Q(y,4\sqrt{n}\ell)}\nu(z)^{-p'/p}\,dz\bigg)^{1/{p'}}\notag\\
&\leq C\cdot w\big(Q(y,4\sqrt{n}\ell)\big)^{1/{\alpha}-1/p-1/s}\big\|f\cdot\chi_{Q(y,4\sqrt{n}\ell)}\big\|_{L^p(\nu)}.
\end{align}
We now turn to estimate the second term $II(y,\ell)$. From the definition \eqref{frac}, it then follows that for any $x\in Q(y,\ell)$,
\begin{equation*}
\begin{split}
\big|I_{\gamma}f_2(x)-(I_{\gamma}f_2)_{Q(y,\ell)}\big|
&=\bigg|\frac{1}{|Q(y,\ell)|}\int_{Q(y,\ell)}\big[I_{\gamma}f_2(x)-I_{\gamma}f_2(z)\big]\,dz\bigg|\\
&=\bigg|\frac{C}{|Q(y,\ell)|}\int_{Q(y,\ell)}
\bigg\{\int_{Q(y,4\sqrt{n}\ell)^c}\bigg[\frac{1}{|x-\zeta|^{n-\gamma}}-\frac{1}{|z-\zeta|^{n-\gamma}}\bigg]f(\zeta)\,d\zeta\bigg\}dz\bigg|\\
&\leq\frac{C}{|Q(y,\ell)|}\int_{Q(y,\ell)}
\bigg\{\int_{Q(y,4\sqrt{n}\ell)^c}\bigg|\frac{1}{|x-\zeta|^{n-\gamma}}-\frac{1}{|z-\zeta|^{n-\gamma}}\bigg|\cdot|f(\zeta)|\,d\zeta\bigg\}dz.
\end{split}
\end{equation*}
By the same reason as in the proof of Theorem $\ref{mainthm:5}$ (see \eqref{average}), we can show that for any $x\in Q(y,\ell)$,
\begin{equation}\label{average2}
\big|I_{\gamma}f_2(x)-(I_{\gamma}f_2)_{Q(y,\ell)}\big|
\leq C\sum_{j=2}^\infty\frac{1}{2^j}\cdot\frac{1}{|Q(y,2^{j+1}\sqrt{n}\ell)|^{1-{\gamma}/n}}\int_{Q(y,2^{j+1}\sqrt{n}\ell)}|f(\zeta)|\,d\zeta.
\end{equation}
Furthermore, by using H\"older's inequality, the preceding expression in \eqref{average2} can be estimated as follows:
\begin{align}\label{end2.3}
&\frac{1}{|Q(y,2^{j+1}\sqrt{n}\ell)|^{1-{\gamma}/n}}\int_{Q(y,2^{j+1}\sqrt{n}\ell)}|f(\zeta)|\,d\zeta\notag\\ &\leq\frac{1}{|Q(y,2^{j+1}\sqrt{n}\ell)|^{1-{\gamma}/n}}\bigg(\int_{Q(y,2^{j+1}\sqrt{n}\ell)}|f(\zeta)|^p\nu(\zeta)\,d\zeta\bigg)^{1/p}
\bigg(\int_{Q(y,2^{j+1}\sqrt{n}\ell)}\nu(\zeta)^{-p'/p}\,d\zeta\bigg)^{1/{p'}}\notag\\
&=w\big(Q(y,2^{j+1}\sqrt{n}\ell)\big)^{1/{\alpha}-1/p-1/s}\big\|f\cdot\chi_{Q(y,2^{j+1}\sqrt{n}\ell)}\big\|_{L^p(\nu)}\notag\\
&\times\frac{w(Q(y,2^{j+1}\sqrt{n}\ell))^{1/q}}{|Q(y,2^{j+1}\sqrt{n}\ell)|^{1-{\gamma}/n}}
\bigg(\int_{Q(y,2^{j+1}\sqrt{n}\ell)}\nu(\zeta)^{-p'/p}\,d\zeta\bigg)^{1/{p'}},
\end{align}
where the last equality is also due to the fact that $1/{\alpha}-1/p-1/s=-1/q$. It then follows from \eqref{end2.3} and \eqref{U2} that
\begin{equation*}
\begin{split}
&\big|I_{\gamma}f_2(x)-(I_{\gamma}f_2)_{Q(y,\ell)}\big|\leq C\sum_{j=2}^\infty\frac{1}{2^j}\cdot w\big(Q(y,2^{j+1}\sqrt{n}\ell)\big)^{1/{\alpha}-1/p-1/s}\big\|f\cdot\chi_{Q(y,2^{j+1}\sqrt{n}\ell)}\big\|_{L^p(\nu)}\\
&\times\frac{|Q(y,2^{j+1}\sqrt{n}\ell)|^{1/{(r'q)}}}{|Q(y,2^{j+1}\sqrt{n}\ell)|^{1-{\gamma}/n}}
\bigg(\int_{Q(y,2^{j+1}\sqrt{n}\ell)}w(\zeta)^r\,d\zeta\bigg)^{1/{(rq)}}
\bigg(\int_{Q(y,2^{j+1}\sqrt{n}\ell)}\nu(\zeta)^{-p'/p}\,d\zeta\bigg)^{1/{p'}}\\
&\leq C\sum_{j=2}^\infty\frac{1}{2^j}
\cdot w\big(Q(y,2^{j+1}\sqrt{n}\ell)\big)^{1/{\alpha}-1/p-1/s}\big\|f\cdot\chi_{Q(y,2^{j+1}\sqrt{n}\ell)}\big\|_{L^p(\nu)}.
\end{split}
\end{equation*}
Consequently,
\begin{align}\label{IIW}
II(y,\ell)&=\frac{1}{|Q(y,\ell)|}\int_{Q(y,\ell)}\big|I_{\gamma}f_2(x)-(I_{\gamma}f_2)_{Q(y,\ell)}\big|\,dx\\
&\leq C\sum_{j=2}^\infty\frac{1}{2^j}\cdot w\big(Q(y,2^{j+1}\sqrt{n}\ell)\big)^{1/{\alpha}-1/p-1/s}\big\|f\cdot\chi_{Q(y,2^{j+1}\sqrt{n}\ell)}\big\|_{L^p(\nu)}.\notag
\end{align}
Therefore by taking the $L^s(\mu)$-norm of \eqref{end2.1}(with respect to the variable $y$), and then using Minkowski's inequality, \eqref{IW} and \eqref{IIW}, we get
\begin{equation*}
\begin{split}
&\bigg\|\frac{1}{|Q(y,\ell)|}\int_{Q(y,\ell)}\big|I_{\gamma}f(x)-(I_{\gamma}f)_{Q(y,\ell)}\big|\,dx\bigg\|_{L^s(\mu)}\\
&\leq\big\|I(y,\ell)\big\|_{L^s(\mu)}+\big\|II(y,\ell)\big\|_{L^s(\mu)}\\
&\leq C\Big\|w\big(Q(y,4\sqrt{n}\ell)\big)^{1/{\alpha}-1/p-1/s}\big\|f\cdot\chi_{Q(y,4\sqrt{n}\ell)}\big\|_{L^p(\nu)}\Big\|_{L^s(\mu)}\\
&+C\sum_{j=2}^\infty\frac{1}{2^j}\cdot\Big\|w\big(Q(y,2^{j+1}\sqrt{n}\ell)\big)^{1/{\alpha}-1/p-1/s}
\big\|f\cdot\chi_{Q(y,2^{j+1}\sqrt{n}\ell)}\big\|_{L^p(\nu)}\Big\|_{L^s(\mu)}\\
&\leq C\big\|f\big\|_{(L^p,L^s)^{\alpha}(\nu,w;\mu)}+C\big\|f\big\|_{(L^p,L^s)^{\alpha}(\nu,w;\mu)}
\times\sum_{j=2}^\infty\frac{1}{2^j}\\
&\leq C\big\|f\big\|_{(L^p,L^s)^{\alpha}(\nu,w;\mu)}.
\end{split}
\end{equation*}
We end the proof by taking the supremum over all $\ell>0$.
\end{proof}

\end{document}